\documentclass{amsart}
\usepackage{hyperref}

\usepackage{amssymb, amsmath}
\usepackage{mathrsfs}
\usepackage{amscd}
\usepackage{verbatim}
\usepackage{a4wide}

\usepackage{enumerate}

\usepackage[fleqn,tbtags]{mathtools}
\mathtoolsset{showonlyrefs,showmanualtags}

\theoremstyle{plain}
\newtheorem{theorem}{Theorem}[section]
\theoremstyle{remark}
\newtheorem{remark}[theorem]{Remark}
\newtheorem{example}[theorem]{Example}
\theoremstyle{plain}
\newtheorem{corollary}[theorem]{Corollary}
\newtheorem{lemma}[theorem]{Lemma}
\newtheorem{proposition}[theorem]{Proposition}
\newtheorem{definition}[theorem]{Definition}

\numberwithin{equation}{section}

\def\N{{\mathbb N}}

\def\R{{\mathbb R}}
\def\C{{\mathbb C}}

\newcommand{\E}{{\mathbb E}}
\renewcommand{\P}{{\mathbb P}}
\newcommand{\F}{{\mathscr F}}

\renewcommand{\a}{\alpha}
\renewcommand{\b}{\beta}
\newcommand{\g}{\gamma}
\renewcommand{\d}{\delta}
\newcommand{\e}{\varepsilon}
\renewcommand{\l}{\lambda}

\renewcommand{\O}{\Omega}

\newcommand{\Schw}{\mathscr{S}}
\newcommand{\Dom}{\mathsf{D}}
\newcommand{\Ran}{\mathsf{R}}

\newcommand{\calL}{{\mathscr L}}
\newcommand{\n}{\Vert}
\newcommand{\one}{{{\bf 1}}}
\newcommand{\embed}{\hookrightarrow}
\newcommand{\s}{^*}
\newcommand{\lb}{\langle}
\newcommand{\rb}{\rangle}

\newcommand{\limn}{\lim_{n\to\infty}}

\newcommand{\wh}{\widehat}
\newcommand{\wt}{\widetilde}

\newcommand{\OO}{\mathcal{O}}
\newcommand{\Tr}{{\rm Tr}}
\newcommand{\Ext}{{\rm Ext}}
\allowdisplaybreaks

\begin{document}

\title[Maximal $\g$-regularity]{Maximal $\g$-regularity}

\author{Jan van Neerven}
\address{Delft Institute of Applied Mathematics\\
Delft University of Technology \\ P.O. Box 5031\\ 2600 GA Delft\\The
Netherlands} \email{J.M.A.M.vanNeerven@tudelft.nl}

\author{Mark Veraar}
\address{Delft Institute of Applied Mathematics\\
Delft University of Technology \\ P.O. Box 5031\\ 2600 GA Delft\\The
Netherlands} \email{M.C.Veraar@tudelft.nl}

\author{Lutz Weis}
\address{Institut f\"ur Analysis \\
Universit\"at Karlsruhe (TH)\\
D-76128  Karls\-ruhe\\Germany}
\email{Lutz.Weis@kit.edu}

\begin{abstract}
In this paper we prove maximal regularity estimates in ``square function spaces''
which are commonly used in harmonic analysis, spectral theory, and stochastic analysis.
In particular, they lead to a new class of maximal regularity results
for both deterministic and stochastic equations
in $L^p$-spaces with $1<p<\infty$. For stochastic equations, the case $1<p<2$ was not
covered in the literature so far.
Moreover, the ``square function spaces'' allow initial values with the same roughness as in the $L^2$-setting.
\end{abstract}

\date\today

\thanks{The first named author is supported by VICI subsidy 639.033.604
of the Netherlands Organisation for Scientific Research (NWO). The second author
is supported by VENI subsidy 639.031.930
of the Netherlands Organisation for Scientific Research (NWO).
The third named author is supported by a grant from the
Deutsche Forschungsgemeinschaft (We 2847/1-2).}

\maketitle

\section{Introduction}

The notion of maximal $L^p$-regularity plays a key role in the functional analytic approach to
nonlinear evolution equations. A sectorial operator $A$
is said to have {\em maximal $L^p$-regularity} if for all $f\in C_{\rm c}(\R_+;\Dom(A))$
the mild solution $u$ of the inhomogeneous
Cauchy problem
\begin{equation}\label{eq:inhom}
\left\{
\begin{aligned}  u'(t) + Au(t)&  = f(t), \quad t\ge 0, \\ u(0)& =0,
\end{aligned}
\right.\end{equation}
satisfies
\begin{equation}\label{eq:maxreg}
 \n Au \n_{L^p(\R_+;X)} \le C \n f\n_{L^p(\R_+;X)}
\end{equation}
with a finite constant $C$ independent of $f$.
In the presence of maximal $L^p$-regularity,
a variety of techniques are available to solve `complicated' (e.g., quasilinear or time-dependent) nonlinear problems
by reducing them to an `easy' (semilinear) problem. This was shown in the classical papers
\cite{CleLi, Prussalshierboven} which spurred a large body of work,
systematic expositions of which are now available in the monographs \cite{Am, DHP, KuWe}.
The related notion of H\"older maximal regularity is discussed in \cite{Lun}.

In the Hilbert space context, the notion of maximal $L^p$-regularity
goes back to de Simon \cite{deSimon64} and Sobolevskii \cite{Sobol64}, who proved that generators of bounded analytic
$C_0$-semigroups on Hilbert spaces have maximal $L^p$-regularity for $p\in (1, \infty)$.
In Banach space setting, maximal regularity $L^p$-regularity in the real interpolation scale
was considered in the work of Da Prato and Grisvard \cite{DPG}. It was shown by Dore \cite{Dore}
that if a sectorial operator $A$ has maximal $L^p$-regularity for some $1<p<\infty$, then it has
maximal $L^p$-regularity for all $1<p<\infty$ and the semigroup generated by $-A$ is bounded and analytic.
The question whether, conversely, every negative generator of a bounded analytic semigroup
on a Banach space $X$ has maximal $L^p$-regularity became known as the `maximal regularity problem'.
After a number of partial affirmative results by various authors, this problem was finally solved in the negative by
Kalton and Lancien \cite{KaLa}. Around the same time, the third named author showed that a sectorial operator $A$
on a UMD Banach space $X$ has maximal $L^p$-regularity if and only if it is {\em $R$-sectorial} of angle $\sigma\in (0,\pi/2)$, which by definition means that for all $\sigma'\in (0,\pi/2)$ the operator family
\begin{equation}\label{eq:Rsect}
\{\l(\lambda+A)^{-1}: \ \ \lambda\in \C\setminus\{0\}, \  |\arg(z)|> \sigma'\}
\end{equation}
is $R$-bounded \cite{We}.

The aim of this paper is to introduce a `Gaussian' counterpart of maximal $L^p$-regularity,
called {\em maximal $\g$-regularity},
and prove that on any Banach space
a sectorial operator $A$ has maximal $\g$-regularity if and only it is $\g$-sectorial.
As an immediate corollary
we see that in UMD Banach spaces, the notions of maximal $L^p$-regularity
and maximal $\g$-regularity are equivalent. Thus our results make it possible to apply maximal
regularity techniques {\em beyond the
UMD setting.}

In the special case $X = L^q(\mu)$, the norm we consider for maximal $\g$-regularity is equivalent to
the classical square function norm
\begin{equation}\label{eq:sqfc}\|f\|_{L^q(\mu;L^2(\R_+))} =
\Big(\int \Big(\int_{\R_+} |f(t,\xi)|^2 \, dt\Big)^{q/2} \, d\mu(\xi)\Big)^{1/q}.\end{equation}
Such square function norms occur frequently in various areas of analysis, notably in
stochastic analysis (Burkholder's inequalities),
spectral theory ($H^\infty$-functional calculus),
and harmonic analysis (Littlewood-Paley theory).

In the case of a general Banach space $X$, we consider the completion $\g(\R_+;X)$
of the $X$-valued step functions with respect to the norm
\begin{equation}\label{eq:gRX}
 \Big\|\sum_{i=1}^n \frac{\one_{(t_i, t_{i+1})}}{(t_{i+1} - t_i)^{1/2}} x_i\Big\|_{\g(\R_+; X)} :=
\Big\|\sum_{i=1}^n \g_i  x_i\Big\|_{L^2(\O;X)},
\end{equation}where $(\gamma_i)_{i=1}^n$ are standard independent Gaussian random variables
on some probability space $(\Omega,\P)$ (the details are presented in
Section \ref{sec:detgamma}). For $X = L^q(\mu)$, the equivalence of
norms
$$ \n f\n_{\gamma(\R_+;L^q(\mu))} \eqsim \n f\n_{L^q(\mu;L^2(\R_+))}$$
is an easy consequence
of Khintchine's inequality.

The norms introduced in \eqref{eq:gRX} were studied from a function space point of view in \cite{KaWe}.
By the extension procedure of \cite{KaWe}, any bounded operator $T$ on $L^2(\R_+)$
extends canonically to a bounded operator $\wt T$ on $ \g(\R_+;X)$. This makes them
custom made to extend
the classical square function estimates from $H^\infty$-functional calculus and
Littlewood-Paley theory to the Banach space-valued setting. In stochastic analysis,
$\g$-norms have been instrumental in extending the It\^o isometry and Burkholder's inequalities to
the UMD space-valued setting (see \cite{NVW1} and the follow-up work on that paper).

We shall say that a sectorial operator $A$ has maximal $\g$-regularity if for all $f\in C_{\rm c}^\infty((0,\infty);\Dom(A))$
the mild solution $u$ of the inhomogeneous problem \eqref{eq:inhom}
satisfies
\begin{equation}\label{eq:maxreg2}
 \n Au \n_{\g(\R_+;X)} \le C \n f\n_{\g(\R_+;X)}
\end{equation}
with a finite constant $C$ independent of $f$.
An important difference with the theory of maximal $L^p$-regularity consists
in the identification of the trace space. Whereas maximal $L^p$-regularity allows for the treatment of nonlinear
problems with initial values in the space real interpolation space $(X, \Dom(A))_{1-\frac1p, p}$, in the presence
of maximal $\g$-regularity initial values in the complex interpolation space $[X,\Dom(A)]_\frac12$ can be allowed.
A more refined comparison between the two theories will be presented in the final section of this paper.

The stochastic counterpart of maximal $L^p$-regularity has been introduced recently in our paper \cite{NVW10},
where it was shown that if $A$ admits a bounded $H^\infty$-calculus of angle less than $\pi/2$ on a space
$L^q(\mu)$ with $2\le q<\infty$, then $A$ has stochastic maximal $L^p$-regularity for all $2<p<\infty$
(with $p=2$ included if $q=2$). Applications of stochastic maximal $L^p$-regularity to nonlinear
stochastic evolution equations have subsequently been worked out in the paper \cite{NVW11eq}.
For second order uniformly elliptic operators on $L^q(\R^d)$, the basic stochastic maximal $L^p$-regularity
estimate had been obtained earlier by Krylov \cite{Kry94a, Kry, Kry06}, who pointed
out that the restriction to exponents $2\le p<\infty$ is necessary even for $A = -\Delta$.

Here, we shall prove that  if $A$ admits a bounded $H^\infty$-calculus of angle less than $\pi/2$ on a
UMD space $X$ with Pisier's property $(\a)$, then $A$ has stochastic maximal $\g$-regularity.
The class of Banach spaces with the properties just mentioned includes the reflexive scale of
the classical function spaces $L^q(\mu)$, Sobolev spaces, Besov spaces and Hardy spaces.
In particular, we obtain the first stochastic maximal regularity result in $L^q(\mu)$-spaces with $1<q<2$ (see
Corollary \ref{cor:maximregL^2}). As in the deterministic case, a larger trace space is obtained: here, instead of
initial values in $(X, \Dom(A))_{\frac12-\frac1p, p}$ as in \cite{NVW10} we can allow arbitrary initial values in $X$.
Once again, for a more refined
comparison we refer to the final section of this paper.

In the presence of type and cotype, various embeddings of $\g$-spaces to and from suitable
interpolation scales are known to hold.
In applications to nonlinear (stochastic) evolution equations this enables us to work out the
precise (maximal) fractional regularity exponents
of the solution spaces. This is achieved in Sections \ref{sec:see}. To illustrate the usefulness of our techniques,
an application to time-dependent problems is presented in Section \ref{sec:see-time}. The results are applied to a
class of second order uniformly elliptic stochastic PDE in Section \ref{sec:appl}.

This paper continues a line of research initiated in \cite{NVW11eq, NVW10}, the notations of which we follow.
For reasons of self-containedness, an overview of the relevant definitions and preliminary results
is given in the next section.
Unless stated otherwise, all linear spaces are real. Occasionally, when we
use spectral arguments, we pass to complexifications without further notice.
By convention, $\R_+:= [0,\infty)$ denotes the closed positive half-line.
For instance, when we say that a function $u$ on $\R_+$ is locally integrable
we mean that it is integrable on every interval $[0,T]$. We shall write $a\lesssim_{p_1,\dots,p_n} b$
if $a\le Cb$ holds with a constant $C$ depending only on $p_1,\dots,p_n$. We write
$a\eqsim_{p_1,\dots,p_n} b$ when both $a\lesssim_{p_1,\dots,p_n} b$ and $a\gtrsim_{p_1,\dots,p_n} b$ hold.
The domain and range of a linear (possibly unbounded) operator $A$ are denoted by $\Dom(A)$ and $\Ran(A)$, respectively.
\section{Preliminaries}

\subsection{$\g$-Boundedness\label{subsec:Rbdd}}

Let $X$ and $Y$ be Banach spaces and let $(\g_n)_{n\ge 1}$ be Gaussian
sequence (i.e., a sequence of independent real-valued standard Gaussian
random variables).
A family $\mathscr{T}$ of bounded linear operators from $X$ to $Y$
is called {\em $\g$-bounded} if there exists a constant $C\ge 0$ such
that for all finite sequences $(x_n)_{n=1}^N$ in $X$ and
$(T_n)_{n=1}^N$ in ${\mathscr {T}}$ we have
\[ \E \Big\n \sum_{n=1}^N \g_n T_n x_n\Big\n^2
\le C^2\E \Big\n \sum_{n=1}^N \g_n x_n\Big\n^2.
\]
The least admissible constant $C$ is called the {\em $\g$-bound} of
$\mathscr {T}$, notation $\g(\mathscr{T})$.
Clearly, every $\g$-bounded family of
bounded linear operators from $X$ to $Y$ is uniformly bounded and
$\sup_{t\in
\mathscr{T}} \|T\| \le \g(\mathscr{T}) $.
If $X$ and $Y$ are Hilbert spaces, then the converse holds as well
and we have $\sup_{t\in
\mathscr{T}} \|T\| = \g(\mathscr{T})$.

Upon replacing the Gaussian sequence by a Rademacher sequence $(r_n)_{n\geq 1}$
we arrive at the related notion of a {\em $R$-bounded} family of operators.
The {\em $R$-bound} of such a family $\mathscr{T}$ will be denoted by
$R(\mathscr{T})$. A
standard randomization argument shows that every $R$-bounded family
$\mathscr{T}$
is $\g$-bounded and $\g(\mathscr{T})\le R(\mathscr{T})$. Both notions are
equivalent if $Y$ has finite cotype (see \cite[Chapter 11]{DJT}).
We refer to \cite{CPSW,DHP,KuWe} for a detailed discussion.



%

\subsection{The spaces $\g(H,X)$}
Let $H$ be a Hilbert space and $X$ a Banach space.
Let $H\otimes X$ denote the space of finite rank operators from $H$ to $X$.
Each $T\in H\otimes X$ can be represented in the form
$$ T = \sum_{n=1}^N h_n\otimes x_n$$
with $N\ge 1$, $(h_n)_{n=1}^N$ orthonormal in $H$, and $(x_n)_{n=1}^N$ a sequence
in $X$. Here, $h\otimes x$ denotes the operator $h'\mapsto[h',h]_H x$.
We define $\g(H,X)$ as the completion of $H\otimes X$ with respect to the norm
$$ \Big\n \sum_{n=1}^N h_n\otimes x_n\Big\n_{\g(H,X)}^2 := \E \Big\n \sum_{n=1}^N \g_n\otimes x_n\Big\n^2.$$
This norm does not depend on the representation of the operator as long as the sequence
$(h_n)_{n=1}^N$ is chosen to be orthonormal in $H$.
The identity mapping $h\otimes x\mapsto h\otimes x$ extends to a contractive embedding
of $\g(H,X)$ into $\calL(H,X)$. This allows us to view elements of $\g(H,X)$ as
bounded linear operators from $H$ to $X$; the operators arising in this way are called
{\em $\gamma$-radonifying.}

A survey of the theory of $\g$-radonifying operators is presented in \cite{NeeCMA}.

\begin{proposition}[Ideal property]\label{prop:ideal}
 Let $H_1,H_2$ be Hilbert spaces and $X_1,X_2$ Banach spaces.
For all $R\in \calL(H_1, H_2)$, $S\in \g(H_2,X_2)$, and $T\in \calL(X_2,X_1)$ one has
$TSR\in \g(H_1,X_1)$ and
\[\|TSR\|_{\g(H_1,X_1)}\leq \|T\|_{\calL(X_2,X_1)} \, \|S\|_{\g(H_2,X_2)} \, \|R\|_{\calL(H_1,H_2)}.\]
\end{proposition}

In the special case when $H = L^2(E,\nu)$, where $(E,\nu)$ is a $\sigma$-finite measure
space, we shall write
\begin{align*}
 \g(E,\nu;X) & = \g(L^2(E,\nu),X),
\\
 \g(E,\nu;H,X) & = \g(L^2(E,\nu;H),X),
\end{align*}
or even $\g(E,X)$ and $\g(E;H,X)$ when the measure $\nu$ is understood. Obviously,
$\g(E;X) = \g(E;\R,X)$. Any simple function
$f: E\to H\otimes X$ induces an element of $L^2(E;H)\otimes X$ in a canonical way,
and under this identification, $\g(E;X)$ and $\g(E;H,X)$ may be viewed as a Gaussian completion
of the $X$-valued, respectively $H\otimes X$-valued, simple functions on $E$. In general, however, not every element
in $\g(E;X)$ or $\g(E;H,X)$ can be represented as an $X$-valued or $\calL(H,X)$-valued function.
Note however, that
for all $T\in \g(E;H,X)$,
$$ \lb T,x\s\rb := T\s x\s$$
can be identified with an element of $L^2(E;H)$ via the Riesz representation theorem.
Moreover,
\begin{equation}\label{eq:L2est}
\|\lb T,x^*\rb\|_{L^2(E;H)} \leq \|T\|_{\g(E;H,X)} \|x^*\|.
\end{equation}

Let $L^1_{\rm fin}(E;X)$ denote the linear space of strongly measurable functions from $E$ into $X$
which are Bochner integrable
on every set of finite measure. A function $f\in L^1_{\rm fin}(E;X)$
{\em defines an element of $\g(E;X)$}, or simply {\em belongs to $\g(E;X)$}, if the
linear operator
$$ T_f: \one_F  \mapsto \int_F f\,d\nu, \quad F\subseteq E, \ \nu(F)<\infty,$$
extends to a bounded linear operator from $L^2(E)$ into $X$ which belongs to $\g(E;X)$.
In this situation we shall simply write $$f\in \g(E;X).$$ Motivated by the above,
for any $T\in \g(E;X)$ and any measurable subset $F\subseteq E$ with $\nu(F)<\infty$
we may define
\begin{equation}\label{eq:def-int}
\int_F T\,d\nu := T(\one_F).
\end{equation}
Likewise, for $T\in \g(E;X)$ we may define
$\one_F T\in \g(E;X)$ by
$$ \one_F T (g) := T(\one_F g), \quad g\in L^2(E),$$
and we have, identifying $L^2(F)$ with a closed subspace of $L^2(E)$ in the natural way,
\begin{equation}\label{eq:restriction}
 \|\one_F T\|_{\g(E;X)} = \|T|_{L^2(F)}\|_{\g(F;X)}.
\end{equation}
Finally, we note that in the case $T$ is represented by a strongly measurable function $f:E\to X$, then
\[T g = \int_E f g \, d\nu, \ \ \ g\in L^2(E),\]
where the integral exists as a Pettis integral (see \cite{DU}).

With these notation we have the following immediate consequence of \cite[Proposition 2.4]{NVW1}:

\begin{proposition}\label{prop:conv} Let
$(F_n)_{n\geq 1}$ be a sequence of measurable subsets in $E$ such that
$\lim_{n\to
\infty}\nu(E\setminus F_n)=0$. Then for all $T\in \g(E;X)$ we have $\lim_{n\to \infty}
\one_{F_n}T = T$ in $\g(E;X)$.
\end{proposition}

The following $\g$-multiplier result, essentially due to \cite{KaWe} (also see \cite[Section 5]{NeeCMA}),
plays a crucial role. Since, its present formulation, the formulation is slightly different,
we show how it can be deduced from the version in \cite{NeeCMA}. As before, $(E,\nu)$ is a
$\sigma$-finite measure space.

\begin{proposition}\label{prop:KW}
Let $X$ and $Y$ be Banach spaces. Let $X_0\subseteq X$ be a dense set.
Let $M:E\to \calL(X,Y)$ be a function with the following
properties:
\begin{enumerate}
\item[\rm(i)] the range $\mathscr{M} := \{M(t): \ t\in E\}$ is $\g$-bounded;
\item[\rm(ii)] for all $x\in X_0$ the function $Mx$ belongs to
$\g(E;Y)$.
\end{enumerate}
Then for all $G\in \g(E;H,X)$ we have
$MG\in \g(E;H,Y)$ and
\begin{equation}\label{eq:gammamultiplier}
\|MG\|_{\g(E;H,Y)}\le \g(\mathscr{M})\|G\|_{\g(E;H,X)}.
\end{equation}
\end{proposition}
\begin{proof}(Sketch) The $\g$-multiplier result presented in \cite{NeeCMA}
 shows that condition (i) implies that $MG$ is well defined as an element of $ \gamma_\infty(E;H,Y)$, the Banach
space of all $\g$-summing operators from $L^2(E;H)$ to $Y$, and that the estimate \eqref{eq:gammamultiplier}
holds. For elements $G \in \g(E;H,X)$
which are linear combinations of elements of the form $(\one_F \otimes h)\otimes x_0$
with $x_0\in X_0$, condition (ii) guarantees that $MG$ does actually belongs to
$\g(E;H,Y)$. Since such $G$ are dense in $\g(E;H,X)$, the general case
follows by approximation.
\end{proof}

By a theorem of Hoffmann-J{\o}rgensen and Kwapie\'n,
condition (ii) is automatically fulfilled if $Y$ does not contain a copy of
$c_0$ (see \cite[Theorem 4.3]{NeeCMA}).
If $E$ is a separable metric space and $M:E \to \calL(X,Y)$ is strongly continuous,
the $\g$-boundedness condition (i) is also necessary for the
above statement to hold (see \cite{KaWe}).

As a special case of Proposition \ref{prop:KW} we note that for all $m\in L^\infty(E)$
and $f \in \g(E;X)$ we have $mf \in \g(E;X)$ and
\begin{equation}\label{eq:gammapointwisemult}
\|m f\|_{\g(E;X)} \leq \|m\|_{L^\infty(E)} \|f\|_{\g(E;X)}.
\end{equation}

The next proposition can be found (for $H=\R$)
in \cite{KaWe}; see also \cite[Proposition 13.9]{NeeCMA}.
\begin{proposition}\label{prop:C1}
Let $H$ be a Hilbert space, $X$ a Banach space, and let $a<b$ be real numbers.
If $\phi:(a,b)\to \g(H,X)$ is continuously differentiable
and \[\int_a^b (s-a)^{\frac{1}{2}} \|\phi'(s)\|_{\g(H,X)} \, ds<\infty,\]
then $\phi\in \g(a,b;H,X)$ and
\[\|\phi\|_{\g(a,b;H,X)}\leq (b-a)^{\frac{1}{2}} \, \n\phi(b-)\n +  \int_a^b (s-a)^{\frac{1}{2}} \|\phi'(s)\|_{\g(H,X)} \, ds.\]
\end{proposition}

For the definitions of type, cotype,
we refer to \cite{DJT, LiTz}. We recall some facts that will be used frequently:

\begin{itemize}
 \item All Banach spaces have type $1$ and cotype $\infty$;
 \item A Banach space is isomorphic to a Hilbert space if and only if it has type $2$ and cotype $2$;
 \item If $X$ has type $p$ (cotype $q$) then it has type $p'$ for all $p'\in [1,p]$ (cotype $q'$ for all $q'\in [q, \infty]$).
 \item $L^p$-spaces, with $1\le p<\infty$, have type $p\wedge 2$ and cotype $p\vee 2$.
\end{itemize}

The next example gives a useful square function characterisation for $\g(E;X)$
in the case of Banach function spaces $X$ with finite cotype.
\begin{example}
Let $(E,\nu)$ be a $\sigma$-finite measure space and let $X$ a Banach function space with
finite cotype. Then the mapping $I:X(L^2(E))\to
\gamma(E;X)$ given by  $I(x\otimes f)g  := [f,g] x$
defines an isomorphism of Banach spaces. In particular,
for all $\nu$-simple functions $\phi:E\to X$ one has
\begin{equation}\label{eq:gammaBfunc}
\|\phi\|_{\g(E;X)} \eqsim_{E,\nu} \Big\|\Big(\int_E |\phi|^2\, d\nu\Big)^{\frac{1}{2}}\Big\|_X.
\end{equation}
\end{example}

The Fourier-Plancherel transform
\[
\wh{f}(\xi) = \int_{\R^d} f(x)e^{-ix\cdot\xi}\,dx, \quad
\xi\in\R^d,\]
initially defined for functions $f = \sum_{n=1}^N g_n\otimes x_n$ in $L^2(\R^d)\otimes X$
by
\begin{align}\label{eq:Fisometry} \wh{g\otimes x} := \wh{g} \otimes x, \quad g\in L^2(\R^d), \ x\in X,
\end{align}
has a unique extension to a isomorphic isomorphism on $\g(\R^d;X)$. Indeed, identifying
a function $f\in L^2(\R^d)\otimes X$ with the corresponding finite rank
operator $T_f$ in $\g(\R^d),X)$, this is evident from the
representation
$$T_{\wh{f}} = T_f\circ \F^*,$$
where $\F: L^2(\R^d)\to L^2(\R^d)$ is the Fourier-Plancherel transform $f\mapsto \wh{f}$
and $\F^*$ is its Banach space adjoint with respect to the duality pairing
$$\lb g,h\rb_{L^2(\R^d)} = \int_{\R^d} g(x)h(x)\,dx.$$

\begin{remark} Notice that:
 \begin{enumerate}
  \item[\rm(i)] we do not normalise the Fourier-Plancherel transform so as to become an
isometry; this would have the disadvantage of introducing constants $\sqrt{(2\pi)^{d}}$
in most of the formulas below;
  \item[\rm(ii)] in the above duality pairing we
do not take complex conjugates in the second argument; only in this way does the
identity $T_{\wh{f}} = T_f\circ \F^*$ hold true.
 \end{enumerate}
\end{remark}

For $s\in \R$ and an open set $\mathcal{O}\subseteq \R^d$ we write
\[\g^s(\mathcal{O},X) := \g(H^{-s}(\mathcal{O});X),\]
where for each $\alpha\in \R$, $H^{\a}(\mathcal{O})$ denotes the usual Bessel potential space.
For $\mathcal{O} = \R^d$ we have the following characterization of
$\g^s(\R^d;X)$. We write $\mathscr{S}(\R^d)$ for the class of Schwartz functions
on $\R^d$.

\begin{proposition}\label{prop:equivgammas}
Let $X$ be a Banach space. For any $f\in \mathscr{S}(\R^d)\otimes X$ we have
equivalences of norms
\begin{align*}
\|f\|_{\g^s(\R^d;X)} & \eqsim \|(1-\Delta)^{s/2} f\|_{\g(\R^d;X)} \\ & \eqsim
\|\xi\mapsto (1+\xi^2)^{s/2} \hat{f}(\xi)\|_{\g(\R^d;X)}
\intertext{with constants depending only on $d$.
If $s\geq 0$, then we have the further equivalences}
& \eqsim \|f\|_{\g(\R^d;X)} +
\sum_{k=1}^d\|D^{s}_kf\|_{\g(\R^d;X)}
\\ &  \eqsim \|\hat{f}\|_{\g(\R^d;X)} + \sum_{k=1}^d\|(i\xi_k)^s
\hat{f}(\xi)\|_{\g(\R^d;X)}
\end{align*}
with constants depending only on $d$.
\end{proposition}
\begin{proof}
Let us fix an arbitrary $f\in \mathscr{S}(\R^d)\otimes X$.
To prove the first equivalence of norms,
note that since $(1-\Delta)^{s/2}:L^2(\R^d)\to
H^{-s}(\R^d)$ is bounded it follows from the right ideal property
that $g_s:= (1-\Delta)^{s/2} f$ belongs to $\g(\R^d;X)$ and
\[\|g_s\|_{\g(\R^d;X)}\lesssim \|f\|_{\g^s(\R^d;X)}.\]
The reverse estimate can be proved in the same way, now using that
$(1-\Delta)^{-s/2} : H^{-s}(\R^d)\to L^2(\R^d)$ is bounded.
The second norm equivalence follows from \eqref{eq:Fisometry} and
$\F[(1-\Delta)^{s/2} f](\xi) = (1+|\xi|^2)^{s/2} \hat{f}(\xi)$.

Suppose now that $s\ge 0$. Fix $k\in \{1, \ldots, d\}$. Note that
\[1\leq (1+|\xi|^2)^{s/2},  \ \ \text{and} \ \ |(i\xi_k)^{s}|\leq
(1+|\xi|^2)^{s/2}.\]
Since the function $m_k(\xi) = {(i\xi_k)^{s}}/{(1+|\xi|^2)^{s/2}}$
is bounded, by \eqref{eq:gammapointwisemult} we
obtain
\begin{align*}
\|\xi\mapsto (i\xi_k)^s \hat{f}(\xi)\|_{\g(\R^d;X)} & = \|\xi\mapsto m_k(\xi)
(1+|\xi|^2)^{s/2} \hat{f}(\xi)\|_{\g(\R^d;X)} \\ & \leq \|m_k\|_{L^\infty(\R^d)}
\|\xi\mapsto (1+|\xi|^2)^{s/2} \hat{f}(\xi)\|_{\g(\R^d;X)}.
\end{align*}
The reverse estimate can be proved in the same way, now using the pointwise
multiplier $$m(\xi) = (1+|\xi|^2)^{s/2} \Big(1+\sum_{k=1}^d|(i\xi_k)^{s}|\Big)^{-1}.$$
Finally, the equivalence of the last two norms follows from \eqref{eq:Fisometry} and
the identity $\F[D^s_k
f](\xi) = (i\xi_k)^s \hat{f}(\xi)$.
\end{proof}

For any $s\in\R$, $\mathscr{S}(\R^d)\otimes X$ is dense in $H^s(\R^d;X)$.
Indeed, this follows from the density of $\mathscr{S}(\R^d)$ in $H^s(\R^d)$
and the density of $H^s(\R^d)\otimes X$ in $\g(H^s(\R^d);X)$.
With this in mind, the first equivalence of norms states that the operator
$(1-\Delta)^{s/2}$ extends to an
isomorphism from $\g(\R^d;X)$
onto $\g^s(\R^d;X)$ (with inverse $(1-\Delta)^{-s/2}$).
The other equivalences can be interpreted similarly.

The next result will only be used for dimension $d=1$. We refer the reader to \cite{Tr1} for details on the Besov space $B^{s}_{p,q}(\mathcal{O};X)$.

\begin{proposition}[$\gamma$-Besov-embedding]\label{prop:gammsobolev}
Let $X$ be a Banach space, $s\in \R$, $p\in [1, 2]$ and $q\in [2, \infty]$. Let
$\mathcal{O}\subseteq \R^d$ be a smooth domain.
\begin{enumerate}
\item[\rm(i)] If $X$ has type $p$, then we have a natural continuous embedding
\[B^{s+d(\frac{1}{p} - \frac12)}_{p,p}(\mathcal{O};X)\hookrightarrow
\g^s(\mathcal{O};X).\]
\item[\rm(ii)] If $X$ has cotype $q$, then we have a natural continuous embedding
\[\g^s(\mathcal{O};X)\hookrightarrow B^{s+d(\frac{1}{q} -
\frac12)}_{q,q}(\mathcal{O};X).\]
\end{enumerate}
\end{proposition}
\begin{proof}
This follows from \cite[Corollary 2.3]{KNVW} and the boundedness of the
extension operator from $B^{s+d(\frac{1}{q} -
\frac12)}_{q,q}(\mathcal{O};X)$ to $B^{s+d(\frac{1}{q} -
\frac12)}_{q,q}(\R^d;X)$.
\end{proof}

\begin{remark}\label{rem:improvedembedding}
The following results can be found in \cite{V:Bessel} and improve on Proposition
\ref{prop:gammsobolev} in certain settings.
\begin{enumerate}
\item[\rm(i)] If $X$ is a $p$-convex Banach lattice with $p\in (1, 2]$, then in
Proposition \ref{prop:gammsobolev} (1) the space $B^{s+d(\frac{1}{p} -
\frac12)}_{p,p}(\mathcal{O};X)$ can be replaced by $H^{s+d(\frac{1}{p} -
\frac12),p}(\mathcal{O};X)$. The same holds if $X$ is a Banach space of type $2$ and then the space $H^{s,2}(\mathcal{O};X)$ embeds in $\g^s(\mathcal{O};X)$.
\item[\rm(ii)] If $X$ is a $q$-concave Banach lattice with $q\in [2,\infty)$, then in
Proposition \ref{prop:gammsobolev} (2) the space $B^{s+d(\frac{1}{q} -
\frac12)}_{q,q}(\mathcal{O};X)$ can be replaced by $H^{s+d(\frac{1}{q} -
\frac12),q}(\mathcal{O};X)$. The same holds if $X$ is a Banach space of cotype
$2$ and then $\g^s(\mathcal{O};X)$ embeds in $H^{s,2}(\mathcal{O};X)$.
\end{enumerate}
\end{remark}

The next result can be seen as a $\gamma$-Hardy inequality.
\begin{proposition}\label{prop:Hardy}
Let $X$ be a Banach space. For all $\alpha>0$ and $f\in \g(\R_+, \sigma^{-2\alpha + 1} d\sigma;X)$,
\begin{equation}\label{eq:HardyYoung}
\Big\|\sigma\mapsto \sigma^{-\alpha - \frac12} \int_0^\sigma f(t) \, \, dt \Big\|_{\g(\R_+;X)}
\leq \alpha^{-1} \big\|\sigma\mapsto \sigma^{-\alpha + \frac12} f(\sigma)\, \big\|_{\g(\R_+;X)}.
\end{equation}
\end{proposition}

\begin{proof} One way to prove this result is to observe that the corresponding inequality
holds with $\g(\R_+;X)$ replaced by $L^2(\R_+)$ and then to invoke the $\g$-extension theorem of
\cite{KaWe}. A simple direct proof runs as follows. {\em It suffices to consider step functions $f$.}
Let $u(\sigma) = \int_0^\sigma f(t) \, \, dt$. Writing $\sigma^{-\alpha - \frac12}  u(\sigma)
= \sigma^{-\alpha + \frac12}  \int_0^1 u'(t \sigma) \, dt$
and taking $\g$-norms on both sides,
\begin{align*}
\|\sigma\mapsto \sigma^{-\alpha - \frac12}  u(\sigma)\|_{\g(\R_+;X)} &= \Big\|\sigma\mapsto \sigma^{-\alpha + \frac12}  \int_0^1 u'(t \sigma) \, dt\Big\|_{\g(\R_+;X)}
\\ & \leq \int_{0}^1 \|\sigma\mapsto \sigma^{-\alpha + \frac12}  u'(t \sigma)\|_{\g(\R_+;X)}\, dt
\\ & = \int_{0}^1 t^{\alpha-1} \, dt \, \|s\mapsto s^{-\alpha + \frac12}  u'(s)\|_{\g(\R_+;X)}
\\ & = \alpha^{-1} \|s\mapsto s^{-\alpha + \frac12}  f(s)\|_{\g(\R_+;X)}.
\end{align*}
\end{proof}

\subsection{Operators with a bounded $H^\infty$-calculus}
In this section we recall some known connections between $H^\infty$-functional calculi and
$\gamma$-radonification. At the same time, this section serves to fix notations and terminology.
We refer the reader to \cite{KuWe} for an in-depth treatment of these matters;
for more on $H^\infty$-calculi the reader may also wish to consult \cite{DHP,Haase:2}.

For $\theta\in (0,\pi)$ we set
$ \Sigma_\theta = \{z\in \C\setminus\{0\}: \ |\arg(z)| < \theta\},$
where the argument is taken in $(-\pi,\pi)$.
A closed densely defined linear operator
$(A, \Dom(A))$ on a Banach space $X$ is said to be
{\em sectorial of type $\sigma\in (0,\pi)$} if it is injective and has dense range,
its spectrum is contained in $\overline{\Sigma_\sigma}$,
and for all $\sigma'\in(\sigma,\pi)$ the set
$$ \big\{z(z+A)^{-1}: \ z\in \C\setminus\{0\}, \  |\arg(z)|> \sigma'\big\}$$
is uniformly bounded.
If infimum of all $\sigma\in (0,\pi)$ such that sectorial of type
$\sigma$ is called the {\em sectoriality angle} of $A$.
The operator $A$ is said to be {\em $\g$-sectorial of type
$\sigma$} if $A$ is sectorial of type
$\sigma$ and the set
$\{z(z+A)^{-1}: \ z\in \C\setminus\{0\}, \  |\arg(z)|> \sigma'\}$ is $\g$-bounded
for all $\sigma'\in (\sigma,\pi)$.
The {\em $\g$-sectoriality angle} of $A$ is defined analogously.

As is well known, if $A$ is a sectorial operator of type $\sigma\in (0,\frac12\pi)$,
then $-A$ generates
a strongly continuous bounded analytic semigroup $S = (S(t))_{t\ge 0}$.
If $A$ is $\gamma$-sectorial of type $\sigma\in (0,\frac12\pi)$,
then the family $\{S(t): t\ge 0\}$ is $\g$-bounded \cite[Theorem 2.20]{KuWe}.

Let $H^\infty(\Sigma_\theta)$ denote the Banach space of all bounded analytic functions
$f:\Sigma_\theta\to \C$, endowed with the supremum norm, and let
$H_0^\infty(\Sigma_\theta)$ denote the linear subspace of all
$f\in H^\infty(\Sigma_\theta)$ for which there exists $\e>0$ and $C\ge 0$ such that
$$ |f(z)| \le \frac{C|z|^\e}{(1+|z|)^{2\e}}, \quad z\in \Sigma_\theta.$$
If $A$ is sectorial of type $\sigma_0\in (0,\pi)$, then for all $\sigma\in (\sigma_0,\pi)$ and
$f\in H_0^\infty(\Sigma_{\sigma})$ we may define the bounded operator $A$
by the Dunford integral
$$ f(A) = \frac1{2\pi i}\int_{\partial \Sigma_{\sigma}} f(z) (z+A)^{-1}\,dz.$$

A sectorial operator $A$ of type $\sigma_0\in (0,\pi)$ and let $\sigma\in (\sigma_0,\pi)$.
is said to have a {\em bounded $H^\infty$-calculus of type $\sigma$} (briefly, $A$ has a
{\em  bounded $H^\infty(\Sigma_{\sigma})$-calculus}) if there is a constant $M\ge 0$
such that for all $f\in H_0^\infty(\Sigma_{\sigma})$ we have
$$ \n f(A)\n \le M\n f\n_{H^\infty(\Sigma_\sigma)}.$$
The infimum of all $\sigma\in (0,\pi)$ such that $A$ has a
bounded $H^\infty(\Sigma_{\sigma})$-calculus is called the {\em $H^\infty$-angle} of $A$.

If $A$ has a bounded $H^\infty(\Sigma_{\sigma})$-calculus, there is a canonical
way to extend the mapping $f\mapsto f(A)$ to a bounded algebra homomorphism from
$H^\infty(\Sigma_\sigma)$ to $\calL(X)$ (of norm $\le M$).
We refer to the lecture notes \cite{KuWe} and the book \cite{Haase:2} for a comprehensive
treatment.

The following result is taken from \cite[Theorem 5.3]{KWcalc}.

\begin{proposition}\label{prop:functionalcalcgammasect}
Let $X$ be a Banach space with property $(\Delta)$. If $A$ has a bounded $H^\infty$-calculus of angle $\sigma$,
then $A$ is
$\g$-sectorial of the same angle $\sigma$.
\end{proposition}

Every UMD space and every Banach space with property $(\alpha)$ has property $(\Delta)$. Moreover, every Banach
space with property $(\Delta)$ has finite cotype. In particular, any Banach space which is isomorphic to a
closed subspace of a space $L^p$ with $p\in [1, \infty)$ has property $(\Delta)$.
For details we refer to \cite{KWcalc}.

From the point of view of
evolution equations, the most interesting class of operators with a bounded
$H^\infty$-calculus of angle $<\pi/2$ consists of
uniformly elliptic operators. Under mild boundedness and smoothness assumptions on the
coefficients, for all $1<p<\infty$ these operators admit a bounded $H^\infty(\Sigma_{\sigma})$-calculus
with $\sigma\in (0,\pi/2)$ on $L^p(\R^d)$, and on $L^p(\mathcal{O})$ with respect to various
boundary conditions if $\mathcal{O}\subseteq \R^d$ is a smooth domain (see \cite{DDHPV, KKW}
and references therein).
Another class of examples can be deduced from Dore's result: any sectorial operator $A$ of
type $\sigma_0\in (0,\pi)$, has a bounded $H^\infty(\Sigma_{\sigma})$-calculus on the real
interpolation space $D_A(\alpha, p)$ for all $\alpha>0$,
$p\in [1, \infty]$ and $\sigma>\sigma_0$ (see \cite{Haase:2}).

The following result is a consequence of \cite[Theorem 7.2, Proposition 7.7]{KaWe}.
It extends McIntosh's classical square function estimates for the Hilbert space
case (see \cite{McI}). The fact that no finite cotype assumption is needed follows by a careful
examination of the proof.

To avoid assumptions on the geometry of Banach spaces under consideration we consider the set
 $$X^\sharp:= \overline{\Dom(A^*)} \cap \overline{\Ran(A^*)}.$$ We denote by $A^{\sharp}$ the part of $A^*$
in $X^\sharp$ (see \cite[Section 15]{KuWe} for details).
\begin{proposition}\label{prop:Hinftysquare}
Let $A$ have a bounded
$H^\infty(\Sigma_{\sigma})$-calculus with $\sigma\in (0,\pi)$ on an arbitrary  Banach space $X$.
Then
for all
$\sigma'\in (\sigma,\pi)$ and all nonzero $\varphi\in H^\infty_0(\Sigma_{\sigma'})$,
\begin{align*}
\|\varphi(t A^\sharp) x^\sharp \|_{\g(\R_+,\frac{dt}{t};X)^*} &\lesssim
\|x^\sharp\|,\quad x^\sharp\in X^\sharp,
\\ \|\varphi(t A) x\|_{\g(\R_+,\frac{dt}{t};X)}& \gtrsim \|x\|,\quad x\in X.
\end{align*}
If $X$ has finite cotype, then we also have
\[
 \|\varphi(t A) x \|_{\g(\R_+, \frac{dt}{t};X)} \lesssim
\|x\|, \quad x\in X.
\]
In these inequalities the implicit constants are independent of $x^\sharp$ and $x$.
\end{proposition}

\section{Maximal $\g$-regularity\label{sec:detgamma}}

Let $-A$ generate a strongly continuous semigroup on a Banach space $X$ and let
$f\in \g(\R_+;X)$.
A locally integrable function $u:\R_+\to X$ is called a {\em weak solution} of
the Cauchy problem
\begin{equation}\label{eq:CP}
\left\{\begin{aligned} u' +  A u  & =  f \ \ \text{on} \ \R_+,
\\ u(0) & =  0,
\end{aligned}\right.
\end{equation}
if for all $t\in (0,\infty)$ and $x^*\in \Dom(A^*)$
\[\lb u(t), x^*\rb
+ \int_0^t \lb u(s), A^* x^*\rb \, ds =
\int_0^t \lb f, x^*\rb(s)\, ds.\]
Note that $\lb f, x^*\rb$ is well defined as an element of $L^2(\R_+)$. It follows
from \cite{Ball77} that weak
solutions, whenever they exist, are unique.

We shall be interested in regularity properties of weak solutions in the
situation when $A$ is a sectorial operator.

\begin{definition}
Let $A$ be a sectorial operator of angle $\sigma\in [0,\frac12\pi)$
and denote by $S$ the bounded analytic semigroup generated by $-A$.
We say that $A$ has {\em maximal $\g$-regularity} if
for all $f\in  C^\infty_{\rm c}(0,\infty; \Dom(A))$
the convolution $u = S*f$ satisfies $Au \in \g(\R_+;X)$ and
\begin{align}\label{eq:gammaDeterministicEst}
\|A u\|_{\g(\R_+;X)}  \le C \|f\|_{\g(\R_+;X)},
\end{align}
with constant $C$ independent of $f$.
\end{definition}

Note that for all
 $f\in C^\infty_{\rm c}(0,\infty; \Dom(A))$ the convolution $u = S*f$
takes values in
 $\Dom(A)$, so the above definition is meaningful. It is easy to check that,
in this situation, $u$ is the unique weak solution of \eqref{eq:CP} and in fact for all $t>0$ we have
\begin{equation}\label{eq:strsol}
u(t) + \int_0^t Au(s)\,ds = \int_0^t f(s)\,ds.
\end{equation}

The space $C^\infty_{\rm c}(0,\infty; \Dom(A))$ is dense in $\g(\R_+;X)$.
Hence if $A$ has maximal $\g$-regularity, the mapping
\begin{align}\label{eq:A-conv} f\mapsto Au = AS * f
\end{align}
admits a unique bounded
extension to $\g(\R_+;X)$.
Note that we do not claim that for general $f\in \g(\R_+;X)$ the convolution $S*f$
can represented by a function which takes values in $\Dom(A)$ almost everywhere.

Differentiating the identity \eqref{eq:strsol} with respect to $t$, we find that if $A$ has maximal
$\g$-regularity, then for all $f\in C^\infty_{\rm c}(0,\infty; \Dom(A))$ we have $u' = -Au + f\in \g(\R_+;X)$
and $$ \n u'\n _{\g(\R_+;X)}  \le C \|f\|_{\g(\R_+;X)},$$
with constant $C$ independent of $f$. As a consequence, also the mapping
$$f\mapsto u' = (S*f)'$$
admits a unique bounded extension to $\g(\R_+;X)$.

\begin{proposition}\label{prop:weak}
Let $A$ be a sectorial operator of angle
$<\pi/2$ on a Banach space $X$. If $A$ has maximal $\g$-regularity, then for all $f\in \g(\R_+;X)$ there exists a unique
weak solution $u$ to \eqref{eq:CP}.
This solution $u$ belongs to $C([0,T];X)$ and there exists a constant $C$, independent of $f$ and $T$, such that
$$ \n u\n_{C([0,T];X)} \le C\sqrt{T} \n f\n_{\g(\R_+;X)}.$$ \end{proposition}
\begin{proof}
The uniqueness has already been observed. To prove the existence, we use an approximation argument.
Let $f\in \g(\R_+;X)$. Choose a sequence $(f_n)_{n\geq 1}$ in $C^\infty_{\rm c}(0,\infty; \Dom(A))$
such that $\limn f_n = f$ in $\g(\R_+;X)$. For each $n\geq 1$, let $u_n = S*f_n$.  By the maximal
$\g$-regularity of $A$, we obtain that $(A u_n)_{n\geq 1}$ and $(u_n')_{n\geq 1}$ are Cauchy
sequences in $\g(\R_+;X)$, and hence convergent to $v$ and $w$ in $\g(\R_+;X)$ respectively.
Fix $T\in \R_+$ and $t\in [0,T]$. For all $x^*\in X^*$ one has
\begin{align*}
|\lb u_n(t) - u_m(t), x^*\rb|& \stackrel{\rm(a)}{\leq} \|\lb A u_n - A u_m, x^*\rb\|_{L^1(0,t)} + \|\lb f_n-f_m, x^*\rb\|_{L^1(0,t)}
\\ & \stackrel{\rm(b)}{\leq} \sqrt{t} \big(\|A u_n - A u_m\|_{\g(0,t;X)} + \|f_n-f_m\|_{\g(0,t;X)}\big)\|x^*\|
\\ & \stackrel{\rm(c)}{\leq} \sqrt{t} (C+1) \|f_n-f_m\|_{\g(0,t;X)}\|x^*\|,
\end{align*}
In (a) we used that $u_n-u_m$ is a weak solution to \eqref{eq:CP} with right-hand side $f_n - f_m$,
in (b) the Cauchy-Schwarz inequality and \eqref{eq:L2est}, and in (c) the inequality \eqref{eq:gammaDeterministicEst}.
Taking the supremum over all $x^*\in X^*$ with $\|x^*\|\leq 1$ and $t\in [0,T]$, it follows that
\[\|u_n - u_m\|_{C([0,T];X)}\leq \sqrt{T} (C+1)\|f_n-f_m\|_{\g(0,t;X)}.\]
It follows that $(u_n)_{n\geq 1}$ is a Cauchy sequence in $C([0,T];X)$ and hence it is convergent to some $u_T\in C([0,T];X)$.
Since $T$ was arbitrary, a uniqueness argument shows that one can find a continuous function
$u:\R_+\to X$ such that $u = u_T$ on $[0,T]$. Finally, we claim that $u$ is a weak solution
to \eqref{eq:CP}. Indeed, this follows from the definition of a weak solution for $u_n$, and the
fact that $\limn u_n=u$ in $C([0,T];X)$, $\limn \lb u_n, A^*x^*\rb = \lb u, A^*x^*\rb$ in $L^1(0,T)$
and $\lim_{n\to \infty}\lb f_n, x^*\rb =  \lb f, x^*\rb$ in $L^1(0,T)$ for each $T<\infty$.
\end{proof}

\medskip
The main result of this section, Theorem \ref{thm:gammaDeterministicEst},
asserts that every $\g$-sectorial operator $A$ of angle
$<\pi/2$ on $X$ has maximal $\g$-regularity. In order to prepare for the proof we make
a couple of preliminary observations. As we have already noted, the $\g$-sectoriality of
$A$ implies that
the set $\mathscr{S} = \{S(t):\ t \ge 0\}$ is $\g$-bounded.
Moreover, by Proposition \ref{prop:C1}, for all $t>0$ and $x\in \Dom(A)$ the reverse orbit
$s\mapsto S(t-s) x$ defines an element of $\g(0,t;X)$. Hence, by Proposition \ref{prop:KW},
for all $f\in \g(0,t;X)$,
\begin{equation}\label{eq:gamma-element}
 s\mapsto S(t-s)f(s)
\end{equation}
is well defined as an element in $\g(0,t;X)$.
We may now define $u: \R_+\to X$ by
$$u(t) := \int_0^t  S(t-s)f(s)\,ds
$$
using the notation introduced in \eqref{eq:def-int}. Recall that the above integral is not defined as a Bochner integral in general.
Likewise, the two integrals in part (i) of the next theorem should be interpreted
in the sense of \eqref{eq:def-int}.

As usual, for $\alpha\in (0,1]$  we denote by $C^{\alpha}(\R_+;X)$
the Banach space of bounded $\alpha$-H\"older continuous functions with values in $X$.
Sometimes we will also write $C^{0}(\R_+;X)$ for the space  $BUC(\R_+;X)$ of bounded uniformly continuous functions.

\begin{theorem}\label{thm:gammaDeterministicEst}
Let $A$ be a $\g$-sectorial operator of angle
$<\pi/2$ on a Banach space $X$. Then $A$ has maximal $\g$-regularity. Moreover, for all
$f\in \g(\R_+;X)$, the convolution $u:=S*f$ satisfies
\begin{enumerate}
 \item[\rm(i)]
$u$ is a weak solution of \eqref{eq:CP} and for all $t\ge 0$ we have
\begin{equation*}
u(t) + \int_0^t A u(s) \, ds = \int_0^t f(s) \, ds.
\end{equation*}
Here, $Au\in \gamma(\R_+;X)$ is defined in the limiting sense as in \eqref{eq:A-conv}.
In particular, $u:\R_+\to X$ is uniformly continuous.
\end{enumerate}
If $0\in \varrho(A)$, then:
\begin{enumerate}
\item[\rm(ii)] (space-time regularity) For all $\theta\in [0,1]$, $u\in
\g^{\theta}(\R_+;\Dom(A^{1-\theta}))$ and
\begin{align*}
\|u\|_{\g^{\theta}(\R_+;\Dom(A^{1-\theta}))}& \eqsim_{A,X,\theta}
\|f\|_{\g(\R_+;X)}.
\end{align*}
\end{enumerate}
\begin{enumerate}
\item[\rm(iii)] (space-time regularity) For all $\theta\in (\frac12 ,1]$ $u\in
C^{\theta-\frac12}(\R_+;\Dom(A^{1-\theta}))$ and
\[\|u\|_{C^{\theta-\frac12}(\R_+;\Dom(A^{1-\theta}))} \lesssim_{A,X,\theta}
\|f\|_{\g(\R_+;X)}.\]
\end{enumerate}
If $0\in \varrho(A)$ and $A$ has a bounded $H^\infty$-calculus of angle $<\pi/2$, then
\begin{enumerate}
\item[\rm(iv)] (trace estimate) $u:\R_+\to \Dom(A^{\frac12})$ is bounded and uniformly continuous,
and we have
\begin{align*}
\|u\|_{BUC(\R_+;\Dom(A^{\frac12}))} \lesssim_{A,X} \|f\|_{\g(\R_+;X)}.
\end{align*}
\end{enumerate}
\end{theorem}
\begin{proof}
We claim that if $f\in C^\infty_{\rm c}(0,\infty; \Dom(A))$, then
for all $\theta\in [0,1]$ we have $D^{\theta}A^{1-\theta}u \in \g(\R_+;X)$
and
\begin{align}\label{eq:gammaDeterministicEstlem}
\|D^{\theta}A^{1-\theta} u\|_{\g(\R_+;X)} \leq C \|f\|_{\g(\R_+;X)},
\end{align}
for some constant $C$ independent of $f$.

To see this let $v :=
D^{\theta}A^{1-\theta} u$. Then
\[\widehat{v}(s) = (is)^{\theta} A^{1-\theta}(is + A)^{-1} \widehat{f}(s).\]
As in \cite[Lemma 10]{KuWePert}
one sees that for all $\theta \in [0,1]$, the operator families
\[\mathcal{T}_1 = \{(is)^{\theta} (is + A)^{-\theta}:s\in \R\setminus \{0\}\},  \ \ \
\text{and} \ \ \ \mathcal{T}_2 = \{A^{1-\theta}(is + A)^{-1+\theta}:s\in \R\setminus \{0\}\} \]
are $\g$-bounded. Hence also $\mathcal{T}_1\mathcal{T}_2$ is $\g$-bounded. In particular,
$$\{(is)^{\theta} A^{1-\theta}(is + A)^{-1}:s\in
\R\setminus \{0\}\}$$ is $\g$-bounded.
Therefore \eqref{eq:Fisometry} and Proposition \ref{prop:KW} imply that
\begin{align*}
\|D^\theta A^{1-\theta} u\|_{\g^{\theta}(\R_+;X)} & = 2\pi\|s\mapsto (is)^{\theta} A^{1-\theta}(is +
A)^{-1} {\widehat f}(s)\|_{\g(\R;X)} \\ & \leq C_{A,\theta}
2\pi \|\widehat{f}\|_{\g(\R;X)} = C_{A,\theta} \|f\|_{\g(\R_+;X)}.
\end{align*}
Maximal $\g$-regularity is obtained by taking $\theta=0$ in \eqref{eq:gammaDeterministicEstlem}.

(i):
In Proposition \ref{prop:weak} we have already seen that $u$ is a weak solution.
Let $(f_n)_{n\geq 1}$ and $(u_n)_{n\geq 1}$ be as in the proof of Proposition \ref{prop:weak}.
Then by \eqref{eq:gammaDeterministicEst} $(A u_n)_{n\geq 1}$ is a Cauchy sequence. Since
$0\in \varrho(A)$, it follows that $(u_n)_{n\geq 1}$ is a Cauchy sequence in $\g(\R_+;\Dom(A))$
and hence convergent
to some $v$ in $\g(\R_+;\Dom(A))$. In the proof of Proposition \ref{prop:weak}, we have seen
that $\limn u_n=u$ in $C([0,T];X)$ for all $T<\infty$. Therefore, one has $v=u$. By
\eqref{eq:def-int}, the required identity holds for each of the $u_n$. The identity for $u$ is obtained
by passing to the limit $n\to \infty$, noting that
$\limn (A u_n)(\one_{[0,t]}) = (A u)(\one_{[0,t]})$ and $\limn f_n(\one_{[0,t]}) = f(\one_{[0,t]})$
in $X$.

(ii): By \eqref{eq:gammaDeterministicEstlem}, applied with $\theta=0$, one sees that $A
u\in \g(\R_+;X)$. Since $0\in \varrho(A)$, this implies that $u\in
\g(\R_+;\Dom(A))$. This proves the result for $\theta=0$. Moreover, $u\in
\g(\R_+;\Dom(A^{1-\theta}))$ for all $\theta\in (0,1]$.  Now the result follows
from \eqref{eq:gammaDeterministicEstlem} and Proposition \ref{prop:equivgammas}.

(iii): By (ii) and Proposition \ref{prop:gammsobolev} with $q=\infty$,
$\g^\theta(\R_+;\Dom(A^{1-\theta})) \hookrightarrow
B^{\theta -\frac12}_{\infty, \infty}(\R_+; \Dom(A^{1-\theta}))$ for all $\theta\in [0,1]$.
If $\theta\in(\frac12 ,1]$, the latter space coincides with  $C^{\theta-\frac12}(\R_+;\Dom(A^{1-\theta}))$
(see \cite[Remark 2.2.2.3 and Corollary 2.5.7]{Tri83}).

(iv): For $f\in C^\infty_{\rm c}(0,\infty; \Dom(A))$ it is clear that $u\in BUC(\R_+;\Dom(A^{\frac{1}{2}}))$;
here we use $0\in \varrho(A)$ to see that the semigroup $S$ is exponentially stable.
Now fix $t\in \R_+$ and  $\varepsilon>0$. Since $X^\sharp$ induces an equivalent norm
on $X$, say $\frac1M\n \cdot\n \le |\!|\!|\cdot |\!|\!|\le \n \cdot\n$ (see \cite[Proposition 15.4]{KuWe}),
we can find $x^*\in X^\sharp$
with $|\!|\!|x^*|\!|\!|= 1$ such that $|\lb A^{\frac{1}{2}}S*f(t), x^*\rb|
\geq (1-\varepsilon)|\!|\!|A^{\frac{1}{2}}S*f(t)|\!|\!|$.
Let $S^\sharp$ be the part of $S^*$ in $X^\sharp$. Then
\begin{align*}
\frac{1-\varepsilon}{M}\|A^{\frac{1}{2}}S*f(t)\| & \leq \int_0^t |\lb A^{\frac{1}{2}}S(t-s) f(s), x^*\rb| \, ds
\\ & = \int_0^t |\lb f(s), (A^\sharp)^{\frac{1}{2}}S^\sharp(t-s)  x^*\rb| \, ds
\\ & \leq \|f\|_{\g(0,t;X)} \|(A^\sharp)^{\frac{1}{2}}S^\sharp(t-\cdot)  x^*\|_{\g(0,t;X)^*}
\\ & \leq \|f\|_{\g(\R_+;X)} \|(A^\sharp)^{\frac{1}{2}}S^\sharp(\cdot)  x^*\|_{\g(\R_+;X)^*}
\\ & \leq C_A \|f\|_{\g(\R_+;X)},
\end{align*}
where in the last step we used Proposition \ref{prop:Hinftysquare}. Since $t\in\R_+$
and $\varepsilon>0$ where arbitrary this yields the required estimate. The case $f\in \g(\R_+;X)$
follows by an approximation argument.
\end{proof}

\begin{remark} \
\begin{enumerate}
\item We expect that in the situation of part (i), $S*f$ does not take values in $\Dom(A)$ almost everywhere
on $(0,\infty)$ and is not differentiable almost everywhere on $(0,\infty)$ in general. However,
if $X$ has cotype $2$, then by Remark \ref{rem:improvedembedding} we have continuous embeddings
$\g^{1}(\R_+;X)\hookrightarrow W^{1,2}(\R_+;X)$ and $\g^0(\R_+;\Dom(A))\hookrightarrow L^2(\R_+;\Dom(A))$, and hence
\[u\in W^{1,2}(\R_+;X)\cap L^2(\R_+;\Dom(A)).\]
\item If $X$ has cotype $q\in [2, \infty]$, then by Proposition \ref{prop:gammsobolev}, for all $\theta\in [0,1]$ we have
\[u\in B^{\theta+\frac1q-\frac12}_{q,q}(\R_+;\Dom(A^{1-\theta}))\]
which improves (iii). A further improvement can be obtained with Remark \ref{rem:improvedembedding}.
\item Part (iv) can be seen as a special case of characterization of traces we will present below in Theorem \ref{thm:traces}.
\end{enumerate}
\end{remark}

\begin{remark}
Under the assumption that $X$ has finite cotype and $A$ has a bounded $H^\infty$-calculus of
angle $<\pi/2$ and $0\in \varrho(A)$, part (iii) of the theorem is optimal in the sense that it
cannot be improved to regularity in $BUC(\R_+;\Dom(A^{\beta}))$ for any $\beta>\frac12$.
To see this let $x\in X$ be arbitrary and define $f_x:\R_+\to X$ by
by $f_x(s) =A^{\frac12} S(s) x$. By Proposition \ref{prop:Hinftysquare},
$f_x\in \g(\R_+;X)$ and $\|f_x\|_{\g(\R_+;X)} \le K\|x\|$ with constant $K$ independent of $x$.
If we had $ S*f \in BUC(\R_+;\Dom(A^{\beta}))$ for some $\beta>\frac12$ and all
$f\in \g(\R_+;X)$, then by a closed graph argument for all $t>0$
we would obtain
\begin{align*}
\ \qquad \|t A^{\beta+\frac12} S(t)x\| & \le \|t A^{\frac12} S(t)x\|_{\Dom(A^\b)}  = \|S*f_x(t)\|_{\Dom(A^\b)}
\\ & \le \|S*f_x(t)\|_{BUC(\R_+;\Dom(A^\b))}\le C\|f_x\|_{\g(\R_+;X)} \leq C K\|x\|.
\end{align*}
Now let $M\ge 1$ and $\omega>0$ be such that $\|S(t)\|\leq M e^{-\omega t}$ for all $t\in \R_+$.
Without loss of generality we may assume $\beta-\frac12 = \frac1N$ for some integer $N\in \N\setminus\{0\}$.
Then for all $t\in (0,1)$,
\begin{align*}
\|A^{\b-\frac12} S(t)x\| & =  \Big\|\int_t^\infty A^{\beta+\frac12} S(s)x \, ds\Big\| \leq \int_t^\infty \|S(s/2)\| \, \|A^{\beta+\frac12} S(s/2)x\| \, ds
\\ & \leq \int_t^\infty  M e^{-\omega s/2} CK   (s/2)^{-1}  \|x\| \, ds  \lesssim (1-\log(t)) \|x\|,
\end{align*}
This is known to be false if $A$ is unbounded. Indeed, from the above estimate one sees that, for all $t\in (0,1)$,
$\|A S(Nt)\|\leq  \|A^{\b-\frac12} S(t)\|^N\lesssim (1-\log(t))^N.$
Hence for all $s\in (0,\frac1N)$ one has $\|A S(s)\|\lesssim (1-\log(s/N))^N$.
In particular, $\limsup_{s\downarrow 0}\|s A S(s)\| = 0$, and this implies that $A$ is bounded (see \cite[Theorem 2.5.3]{Pa}).
\end{remark}

Theorem \ref{thm:gammaDeterministicEst} admits the following converse.

\begin{theorem}\label{thm:convmaxgamma}
Suppose $A$ is a sectorial operator of angle $\sigma\in (0,\pi/2)$ on a Banach space $X$.
If  $A$ has maximal $\g$-regularity and $0\in\varrho(A)$, then $A$ is $\g$-sectorial.
\end{theorem}
\begin{proof}
We claim that
for all Schwartz functions
$f\in \Schw(\R)\otimes \Dom(A)$ one has
\begin{equation}\label{eq:gammaDeterministicEst2}
\|A S* f\|_{\g(\R;X)}\leq C_A \|f\|_{\g(\R;X)}.
\end{equation}
Here $S*f:\R\to \R$ is defined by
\[S*f(t) := \int_{-\infty}^t S(t-s) f(s)\, ds.\]

We first show how the claim can
be applied to obtain the $\g$-sectoriality of $A$.
Let $g\in \Schw(\R)\otimes \Dom(A)$ be arbitrary and set $f = \hat{g}$. From
\eqref{eq:Fisometry} and \eqref{eq:gammaDeterministicEst2} one obtains that
\[\|s\mapsto A (is+A)^{-1} g(s)\|_{\g(\R;X)} \eqsim \|A S*f\|_{\g(\R;X)}
\le C_A\|f\|_{\g(\R;X)} \eqsim C_A\|g\|_{\g(\R;X)}\]
with universal implied constants in the equivalences.
By density, this estimate can be extended to all $g\in \g(\R;X)$.
Now by the converse of Proposition \ref{prop:KW} one sees that
$\{A(is+A)^{-1}: s\in \R \setminus \{0\}\}$ and hence
$\{s (is+A)^{-1}: s\in \R \setminus \{0\}\}$ is $\g$-bounded.
Now the result follows from \cite[Theorem 2.20]{KuWe}.

To prove the claim we adjust an argument in \cite[Theorem 7.1]{Dore}.
Fix $T\in \R$ and $f\in \Schw(\R)\otimes\Dom(A)$. For $t>T$ set
\[U_T f(t) := \int_{-\infty}^T S(t-s) f(s)\, ds  \ \ \text{and} \ \ V_T f(t) := \int_T^t S(t-s) f(s)\, ds.\]
Obviously, $S*f(t) = U_T f(t) + V_T f(t)$.  For $t\geq T+1$ one has
\[A U_T f(t) = \int_{1}^\infty A S(s) \one_{(-\infty,T+s)}(t) f(t-s) \, ds,\]
and one can estimate
\begin{align*}
\|A U_T f\|_{\g(T+1, \infty;X)} & \leq \int_1^\infty \|t\mapsto A S(s) f(t-s)\|_{\g(T+1,T+s;X)}\, ds
\\ & \leq \int_1^\infty \|AS(s)\| \|t\mapsto f(t-s)\|_{\g(T+1,T+s;X)}\, ds
\\ & = \int_1^\infty \|AS(s)\| \|f\|_{\g(T+1-s,T;X)}\, ds
\\ & \leq \n AS(1)\n \int_0^\infty \|S(r)\| \, dr \, \|f\|_{\g(\R;X)}
\\ & = K_A \|f\|_{\g(\R;X)},
\end{align*}
noting that the assumption
$0\in \varrho(A)$ implies the exponential stability of $S$.

On the other hand, if $t>T$, then
\[V_T f(t) = \int_0^{t-T} S(t-T-s) f(s+T)\, ds = S*h(t-T),\]
where $h(s) = f(s+T)\one_{[0,\infty)}(s)$. Hence, by \eqref{eq:gammaDeterministicEst}
applied with $h$ instead of $f$, and observing that $\n g(\cdot-T)\n_{\gamma(T+1,\infty;X)} = \n g\n_{\gamma(1,\infty;X)} \le \n g\n_{\gamma(\R_+;X)}$,
we obtain
\begin{align*}
\|AV_T f\|_{\g(T+1,\infty;X)} & = \|A S*h(\cdot-T)\|_{\g(T+1,\infty;X)}
\\ & \leq \|A S*h\|_{\g(\R_+;X)}\leq C_A \|h\|_{\g(\R_+;X)} \leq C_A \|f\|_{\g(\R;X)}.
\end{align*}
Using Proposition \ref{prop:conv}, we conclude that
\begin{align*}
\|AS*f\|_{\g(\R;X)} & = \lim_{T\to -\infty} \|AS*f\|_{\g(T+1, \infty;X)}
\\ & \leq \lim_{T\to -\infty} \big(\|AU_T f\|_{\g(T+1,\infty;X)}+ \|AV_T f\|_{\g(T+1,\infty;X)}\big)
\\ & \leq (K_A +C_A)\|f\|_{\g(\R;X)}.
\end{align*}
\end{proof}

\begin{corollary}
Let $X$ be a Banach space. Let $A$ be a sectorial operator of angle $<\pi/2$ with $0\in \varrho(A)$.
The following assertions are equivalent:
\begin{enumerate}
\item[\rm(1)] $A$ has maximal $\g$-regularity.
\item[\rm(2)] $A$ is $\g$-sectorial of angle $<\pi/2$.
\end{enumerate}
If, in addition, $X$ is a UMD Banach space, then {\rm (1)} and {\rm (2)} are equivalent to
\begin{enumerate}
\item[\rm(3)] $A$ has maximal $L^p$-regularity for some/all $p\in (1, \infty)$.
\end{enumerate}
\end{corollary}
For the definition of maximal $L^p$-regularity we refer to \cite{We}.

\begin{proof}
(1) $\Leftrightarrow$ (2) holds for any Banach space and follows
from Theorems \ref{thm:gammaDeterministicEst} and \ref{thm:convmaxgamma}.
(3) $\Rightarrow$ (2) holds for any Banach space (see \cite[Section 3.13]{KuWe}
and note that $R$-boundedness implies $\g$-boundedness). Finally (2) $\Rightarrow$ (3) holds in
UMD Banach spaces (see \cite{KuWe, We} and note that in spaces with finite cotype,
$\g$-sectoriality implies $R$-sectoriality; the space $X$, being UMD, has finite cotype).
\end{proof}

Clearly, for every $u\in \g^1(\R_+;X)$ one has $u\in C^{1/2}(\R_+;X)$ and in
particular $\Tr_0 u := u(0)$ exists in $X$ (see Proposition \ref{prop:gammsobolev}). It is
therefore a natural question to characterize the traces of the maximal regularity space
$\g^1(\R_+;X)\cap\g(\R_+;\Dom(A))$. This is achieved in the next theorem and will be proved for sectorial operators of arbitrary angle.

\begin{theorem}[Characterization of traces]\label{thm:traces}
Let $A$ be a $\g$-sectorial operator of angle
$<\pi$ on a Banach space $X$. Assume that $0\in \varrho(A)$ and that $A$ has a bounded $H^\infty$-calculus of angle $<\pi$.
\begin{enumerate}
\item[\rm(i)] The trace map $\Tr_0 u := u(0)$ is bounded from $\g^1(\R_+;X)\cap\g(\R_+;\Dom(A))$  to $\Dom(A^{1/2})$.

\item[\rm(ii)] If  $X$ has finite cotype, then the extension operator $\Ext(x)(t) = (1+tA)^{-1}x$
is bounded from $\Dom(A^{1/2})$ to $\g^1(\R_+;X)\cap\g(\R_+;\Dom(A))$ and
 defines a bounded right-inverse of $\Tr_0$.
\end{enumerate}
\end{theorem}
Note that, as a consequence of (i) and the strong continuity of
the left-translation semigroup $T = (T(t))_{t\ge 0}$ in $ \g^1(\R_+;X)\cap\g(\R_+;\Dom(A))$,
given by $(T(t) u)(s) = u(t+s)$ for $t,s\in \R_+$,
we obtain a continuous embedding
\begin{equation}\label{eq:traceBUC}
\g^1(\R_+;X)\cap\g(\R_+;\Dom(A)) \embed BUC(\R_+;\Dom(A^{1/2})).
\end{equation}

\begin{proof}
(i): By density it suffices to consider functions $u\in C^1_{\rm c}([0,\infty);\Dom(A))$.
Indeed, fix $u\in \g^1(\R_+;X)\cap\g(\R_+;\Dom(A))$. Setting $u(t) = u(-t)$ for $t<0$, we may extend $u$ to a function in $\g^1(\R;X)\cap\g(\R;\Dom(A))$. Multiplying $u$ by a smooth function with compact support it suffices to consider the case where $u$ has compact support. Let $\varphi\in C^\infty_c(\R)$ be a positive function such that $\int \varphi = 1$. Let $\varphi_n(t) = n\varphi(n t)$. Set $u_n = \varphi_n*u$. Then by Proposition \ref{prop:equivgammas}
\[\|u- u_n\|_{\g^1(\R_+;X)} \leq \|u- u_n\|_{\g^1(\R;X)} = \| (1-\hat{\varphi}(\cdot/n))(1+|\cdot|^2)^{1/2} \hat{u}\|_{\g(\R;X)},\]
and the latter converges to zero by \cite[Proposition 2.4]{NVW1} and the fact that $(1+|\cdot|^2)^{1/2} \hat{u}\in \gamma(\R;X)$. Since $n(n+A)^{-1}\to I$ strongly, a further approximation argument yields the required result.

Note that $u \in \gamma(\R_+;D(A))$ and $u'\in \gamma(\R_+;X)$ (for instance by Proposition \ref{prop:C1} or \ref{prop:gammsobolev}).
By Proposition \ref{prop:Hinftysquare}, there is a constant $C$ such that for all $x\in X$ we have
\begin{equation}\label{eq:equiv-norm}
\|x\|\leq C\big\|\sigma\mapsto A^{1/2} (I+\sigma A)^{-1} x\big\|_{\g(\R_+;X)}.
\end{equation}
The method of proof is based on the argument in \cite[Lemmas 11, 12]{DiBlasio}
(see also \cite[Lemma 4.1]{MeySchn} and \cite[Theorem 1.4]{MeyVer2}).
For all $\sigma > 0$ we have
\begin{equation}\label{eq:Tr}
\Tr_0 u = u(0) = \sigma^{-1}\int_0^\sigma u(\tau)\,d \tau - \int_0^\sigma t^{-2} \int_0^t u(t)-u(\tau) \, d\tau  \, dt.
\end{equation}
Therefore, using \eqref{eq:equiv-norm} in which we view $x$ as a constant function of $\sigma$
and substitute for it the right-hand side of \eqref{eq:Tr} which is also constant in $\sigma$,
we obtain the estimate
\[\|\Tr_0 u\|_{\Dom(A^{1/2})}  \leq C (T_1+ T_2),\]
where
\begin{align*}
T_1 &= \Big\|\sigma\mapsto \sigma^{-1}\int_0^\sigma  A(I+\sigma A)^{-1} u(\tau)\,d \tau \Big\|_{\g(\R_+;X)},\\
T_2 &=\Big\|\sigma\mapsto \int_0^\sigma t^{-2} \int_0^t A (I+\sigma A)^{-1} (u(t)-u(\tau)) \, d\tau  \, dt\Big\|_{\g(\R_+;X)}.
\end{align*}
By assumption, the set
$\{(I+\sigma A)^{-1}:\ \sigma\geq 0\}$ is $R$-bounded, and hence $\gamma$-bounded.
Therefore, by Proposition \ref{prop:KW} and Proposition \ref{prop:Hardy} with $\alpha=1/2$,
\begin{align*}
T_1  & \leq C\Big\|\sigma\mapsto \sigma^{-1} \int_0^\sigma A u(\tau) \,d \tau \Big\|
\leq 2C \|A u\|_{\g(\R_+;X)}.
\end{align*}

For estimating $T_2$ note that
\[f(t) := t^{-2} \int_0^t u(t)-u(\tau) \, d\tau = t^{-2} \int_0^t \int_\tau^t u'(s) \, ds \, d\tau =t^{-2} \int_0^t s  u'(s) \, ds.\]
By assumption the set $\{\sigma A (1+\sigma A)^{-1}: \ \sigma\geq 0\}$
is $\gamma$-bounded. Applying Proposition \ref{prop:KW} and Proposition \ref{prop:Hardy} (first with $\alpha=1/2$ and then with $\alpha=3/2$) one obtains that
\begin{align*}
T_2 & \leq C \Big\|\sigma\mapsto \sigma^{-1} \int_0^\sigma f(t)  \, dt\Big\|_{\g(\R_+;X)} \leq 2C \|f\|_{\g(\R_+;X)}
\\ & = 2C \Big\|t\mapsto t^{-2} \int_0^t s  u'(s) \, ds\Big\|_{\g(\R_+;X)} \leq \frac{4 C}{3} \|u'\|_{\g(\R_+;X)}.
\end{align*}

(ii): \ This follows from the fact that $x = (1+0A)^{-1} x$, $\frac{d}{dt} (1+tA)^{-1} = -A(1+tA)^{-1}$
and $\|A(1+tA)^{-1} x\|_{\g(\R_+;X)} \eqsim \|A^{1/2}x\|_{X}$,
for all $x\in \Dom(A^{1/2})$ (see Proposition \ref{prop:Hinftysquare}).
\end{proof}

\section{Stochastic maximal $\g$-regularity\label{sec:stochgamma}}

Let $(\O,\mathcal{A},\P)$ be a probability space endowed with a filtration $\F =
(\F_t)_{t\geq 0}$, which we consider to be fixed throughout the rest of this paper.
An {\em $\F$-cylindrical Brownian motion in $H$} is a bounded linear
operator $W_H: L^2(\R_+;H)\to L^2(\O)$ such that:
\begin{enumerate}[\rm(i)]
\item for all $f\in L^2(\R_+;H)$ the random variable
$W_H(f)$ is centred Gaussian.
\item for all $t\in \R_+$ and $f\in L^2(\R_+;H)$ with support in $[0, t]$, $W_H(f)$ is $\F_t$-measurable.
\item for all $t\in \R_+$ and $f\in L^2(\R_+;H)$ with support in $[t, \infty)$, $W_H(f)$ is independent of $\F_t$.
\item for all $f_1,f_2\in L^2(\R_+;H)$ we have
$ \E (W_H(f_1)\cdot W_H(f_2)) = [f_1,f_2]_{L^2(\R_+;H)}.$
\end{enumerate}
It is easy to see that for all $h\in H$ the process $(W_H(t)h )_{t\ge 0}$ defined by
$$W_H(t)h := W_H(\one_{(0,t]}\otimes h)$$
is an $\F$-Brownian motion $W_H h$ (which is standard if $\n h\n=1$). Moreover, two such
Brownian motions $W_H h_1$  and $W_H h_2$ are independent if and only
if $h_1$ and $h_2$ are orthogonal in $H$.

For a Banach space $E$, let $L^0(\Omega;E)$ denote the vector space of
strongly measurable $E$-valued functions equipped with the (metric) topology induced by convergence in probability,
identifying functions which are equal almost surely. An element $G\in L^0(\Omega;\g(\R_+;H,X))$ is said to be
{\em adapted} (to the filtration $\F$) if for all $t\in \R_+$
and $h\in H$ the random variable $G_{t,h}:\O\to X$ given by $G_{t,h} = G(\one_{[0,t]} \otimes h)$ is
$\F_t$-measurable.
We denote by
$L^0_{\F}(\O;\g(\R_+;H,X))$ the closed subspace of $L^0(\O;\g(\R_+;H,X))$ consisting of its adapted elements.
It coincides with the closure of all adapted elementary step processes
in $L^0(\O;\g(\R_+;H,X))$ (see \cite[Section 2.4]{NVW1}).
We shall write  $L^0_{\F}(\O;\g(\R_+;X)) = L^0_{\F}(\O;\g(\R_+;\R,X))$.
For $p\in (0,\infty)$, the spaces $L_\F^p(\O;\g(\R_+;H,X))$ and $L_\F^p(\O;\g(\R_+;X))$ are defined similarly.

The {\em stochastic integral} with respect to an $H$-cylindrical
Brownian motion $W_H$ of an adapted simple process with values in $H\otimes X$ is defined by
$$ \int_0^t \one_{A\times(a,b]} \otimes (h\otimes x) := \one_A W_H(\one_{(a,b]\otimes h}) \otimes x$$
and linearity; here $0\le a<b<\infty$, $A\in \F_a$, $h\in H$, and $x\in X$.

The following result has been proved in \cite{NVW1} for $p\in (1,\infty)$;
the extension of \eqref{eq:NVW} to $p\in (0,\infty)$ is in \cite{CoxVer2}. Alternatively,
this extension
may be derived from Lenglart's inequality \cite{Lenglart}.

\begin{proposition}[It\^o isomorphism]\label{prop:Ito}
If $X$ is a UMD Banach space, then the mapping
$ G \mapsto \int^\cdot_0 G\,dW_H$ admits a unique extension to a homeomorphism from
$L^0_{\F}(\O;\g(\R_+;H,X))$ onto the space $M_{\rm c}^{\rm loc}(\R_+;X)$ of $X$-valued continuous local martingales.
Moreover, for all
$p\in (0,\infty)$ one has the two-sided estimate
\begin{equation}\label{eq:NVW}
\E \sup_{t\ge 0}\Big\| \int_0^t G\,dW_H\Big\|^p \eqsim_{p,X}
\E\|G\|_{\g(\R_+;H,E)}^p.
\end{equation}
In particular,
by Doob's maximal inequality, for $p\in (1, \infty)$ one has
\[\E \Big\| \int_0^\infty G\,dW_H\Big\|^p \eqsim_{p,X}
\E\|G\|_{\g(\R_+;H,E)}^p.\]
\end{proposition}

Now let $A$ be a sectorial operator of angle $<\pi/2$ on a Banach space $X$.
Our aim is to prove a stochastic $\g$-maximal regularity
result for the stochastic Cauchy problem
\begin{equation}\label{eq:SCP}
\left\{
  \begin{aligned}
dU +  A U \, dt  & = G \, d W_H \ \ \text{on} \ \R_+,
\\ u(0) & = 0.
  \end{aligned}
\right.
\end{equation}
Here, $W_H$ is a cylindrical Brownian motion in a Hilbert space $H$, defined
on a probability space and
$G\in L^0_{\F}(\O;\g(\R_+;H,X))$ is adapted.

A strongly measurable adapted process $U: [0,\infty)\times\Omega \to X$ is called a {\em weak
solution} of \eqref{eq:SCP} if, almost surely, its trajectories are locally Bochner integrable and
for all $t\in (0,\infty)$ and $x^*\in \Dom(A^*)$
almost surely one has
\begin{align}\label{eq:weaksol}
\lb U(t), x^*\rb + \int_0^t \lb U(s), A^*x^*\rb \, ds = \int_0^t G^* x^*\,
d W_H.
\end{align}
Note that $G^* x^* \in L^0_{\F}(\O;L^2(\R_+;H))$. As before, weak solutions are unique.

Let $G:\R_+\times\O\to H\otimes X$ be an adapted step process.
We claim that for all $t>0$ and all $p\in (0,\infty)$ the process
$$s\mapsto S(t-s) G(s)$$ defines an element $L^p_{\F}(\O;\g(0,t;H,X))$.
Indeed, fix $h\in H$, $x\in X$, and $0\leq a<b$.
Fixing an arbitrary $\varepsilon\in (0,\frac12 )$, we write
$S(s) (h\otimes x) = s^{\varepsilon}S(s) f(s)$, where
$f:(a,b)\to \calL(H,X)$ is given by $f(s) = s^{-\varepsilon} h\otimes x$.
By \cite[Example 2.18]{KuWe} $\{s^{\varepsilon} S(s): s\in (a,b)\}$
is $R$-bounded, and since $f\in \g(a,b;H,X)$, it follows from Proposition
\ref{prop:KW} that $s\mapsto S(s) (h\otimes x)\in \g(a,b;H,X)$. Now the
claim follows from an easy substitution argument and taking linear combinations.

In the setting just discussed, Proposition \ref{prop:Ito} implies that the random variable
\[S\diamond G(t)  :=  \int_0^t S(t-s) G(s)\, dW_H(s)\]
is well defined in $L^p(\O;X)$.

\begin{definition} \label{def:gmaxreg}
A sectorial operator $A$ of angle $< \pi/2$ has {\em stochastic maximal $\g$-regularity}
if there exist $p\in (0,\infty)$ and $C \ge 0$
such that for all adapted step processes $G:\R_+\times\O\to H\otimes \Dom(A^{\frac12 })$
we have $A^{\frac{1}{2}} S \diamond G\in L^p(\O;\g(\R_+;X))$ and
\begin{equation}\label{eq:maxgammareg}
 \| A^{\frac{1}{2}} S \diamond G\|_{L^p(\O;\g(\R_+;X))} \le C \| G\|_{ L^p(\O;\g(\R_+;H,X))}.
 \end{equation}
\end{definition}
Here, $A^{\frac{1}{2}} S \diamond G :=  S \diamond A^{\frac{1}{2}}G$ is well defined in view of the preceding discussion.
If $A$ has stochastic maximal $\g$-regularity, the mapping $G\mapsto A^{\frac{1}{2}} S \diamond G$
extends to a bounded linear operator from $L^p_{\F}(\O;\g(\R_+;H,X))$ to $L^p(\O;\g(\R_+;X))$.
As in the previous section, we will write $A^{\frac{1}{2}} S \diamond G$ for this extension general
and keep in mind that this notation is formal; the rigorous interpretation
is in terms of the just-mentioned bounded linear operator.

The above definition evidently depends on the parameter $p$.
In the next proposition, however, we show that, at least for UMD spaces $X$,
stochastic maximal $\g$-regularity is $p$-independent.

\begin{proposition}\label{prop:pindependence}
Let $X$ be a UMD Banach space. If $A$ has stochastic maximal $\g$-regularity,
then for all $q\in (0,\infty)$ there is a constant $C$ such that for all
adapted step processes $G:\R_+\times\O\to H\otimes \Dom(A^{\frac12 })$ one has
\[ \| A^{\frac{1}{2}} S \diamond G\|_{L^q(\O;\g(\R_+;X))} \le C\| G\|_{ L^q(\O;\g(\R_+;H,X))}.\]
\end{proposition}
\begin{proof}
Let $G:\R_+\to H\otimes \Dom(A^{\frac12 })$ be a (deterministic) step function.
In that case, $A^{\frac{1}{2}} S \diamond G$ is a Gaussian random variable with values in $\g(\R_+;X)$.
By Proposition \ref{prop:Ito} applied to the UMD space $\g(\R_+;X)$ and the Kahane-Khintchine inequalities,
for all $t>0$ we have
\begin{equation}\label{eq:omegaind}
\begin{aligned}
\|t\mapsto [s\mapsto \one_{t>s} A^{\frac12} S(t-s) G(s)]\|_{\g(\R_+, dt;H,\g(\R_+,ds;X))}
& \eqsim_{X} \| A^{\frac{1}{2}} S \diamond G \|_{L^2(\O;\g(\R_+;X))}
\\ & \eqsim_{p,X} \| A^{\frac{1}{2}} S \diamond G\|_{L^p(\O;\g(\R_+;X))}
\\ & \lesssim\| G\|_{\g(\R_+;H,X)},
\end{aligned}
\end{equation}
using \eqref{eq:maxgammareg} in the last line; the exponent $p$ is as in Definition \ref{def:gmaxreg}.

Now let $G:\R_+\times \O\to H\otimes \Dom(A^{\frac12 })$ be an adapted step process and
let $q\in (0,\infty)$ be arbitrary.
By Proposition \ref{prop:Ito} applied to the UMD space $\g(\R_+;X)$ and the $\g$-Fubini isomorphism
\cite[Proposition 2.6]{NVW1},
\begin{align*}
\ & \| A^{\frac{1}{2}} S \diamond G\|_{L^q(\O;\g(\R_+;X))}
\\ & \qquad \lesssim_{q,X} \|t\mapsto [s\mapsto  \one_{t>s}
A^{\frac12} S(t-s) G(s)]\|_{\g(\R_+, dt;L^q(\O;\g(\R_+,ds;H,X)))}
\\ & \qquad \eqsim_{q} \|t\mapsto [s\mapsto  \one_{t>s}  A^{\frac12} S(t-s)
G(s)]\|_{L^q(\O;\g(\R_+, dt;\g(\R_+,ds;H,X)))}
\\ & \qquad \lesssim_{p,q,X} \| G\|_{L^q(\O;\g(\R_+;H,X))},
\end{align*}
where in the last step we used \eqref{eq:omegaind} pointwise on $\O$.
\end{proof}

In the next result we will provide sufficient conditions for
stochastic maximal $\g$-regularity under a functional calculus assumption on $A$.
The Banach space $X$ is required to be a UMD space with Pisier's property $(\alpha)$. This property is
equivalent to the assertion that for all non-zero Hilbert spaces $H_1$ and $H_2$,
the mapping $h_1\otimes (h_2\otimes x) \mapsto (h_1\otimes h_2)\otimes x$ induces an isomorphism
of Banach spaces (see \cite{KaWe,NWalpha})
\begin{equation}\label{eq:isomalpha}
\g(H_1,\g(H_2,X)) \simeq \g(H_1\otimes H_2,X).
\end{equation}
The spaces $X =L^q$ have property $(\alpha)$ for all $q\in [1, \infty)$.
If $X$ is isomorphic to a closed subspace of a Banach lattice,
then property $(\alpha)$ is equivalent with finite cotype \cite{Pialpha}.
In particular, every UMD Banach lattice
has property $(\a)$.

In the next theorem we combine Propositions \ref{prop:KW} and
\ref{prop:functionalcalcgammasect} to see that, under the conditions as stated in the theorem, the random variables
$U(t) := S\diamond G(t)$ are well defined in $L^p(\Omega;X)$ for all $t\ge 0$.

\begin{theorem}[Stochastic maximal $\g$-regularity]\label{thm:gammaStochasticEst}
Let $X$ be a UMD Banach space with property $(\alpha)$ and let $p\in (0,\infty)$.
If $A$ has a bounded $H^\infty$-calculus of angle
$<\pi/2$ on $X$, then $A$ has stochastic maximal $\g$-regularity.
Moreover, for all $G\in L^p_{\F}(\O;\g(\R_+;H,X))$ the stochastic
convolution process $U = S\diamond G$ satisfies:
\begin{enumerate}
\item[\rm(i)] (weak solution) $U$ is a weak solution of \eqref{eq:SCP}.
\end{enumerate}
If $0\in \varrho(A)$, then in addition we have:
\begin{enumerate}
\item[\rm(ii)] (space-time regularity) For all $\theta\in [0,\frac12)$, $U\in
L^p(\O;\g^{\theta}(\R_+;\Dom(A^{\frac12-\theta})))$ and
\begin{align*}
\|U\|_{L^p(\O;\g^{\theta}(\R_+;\Dom(A^{\frac12-\theta})))}\lesssim_{A,p,X,\theta}
\|G\|_{L^p(\O;\g(\R_+;H,X))},
\end{align*}
where $\lesssim$ can be replaced by $\eqsim$ if $p\in (1, \infty)$.
\item[\rm(iii)] (trace estimate) $U: \R_+\times \Omega \to X$ is pathwise continuous  and
\begin{align*}
\|U\|_{L^p(\O;BUC(\R_+;X))} \lesssim_{A,p,X} \|G\|_{L^p(\O;\g(\R_+;H,X))}.
\end{align*}
\end{enumerate}
\end{theorem}
\begin{proof}
First we prove that for all $G\in L_\F^p(\O;\g(\R_+;H,X))$
we have
$D^{\theta}A^{\frac12-\theta} U\in L^p(\O;\g(\R_+;X))$ and
\begin{align}\label{eq:gammaStochasticEst}
\|D^{\theta}A^{\frac12-\theta} U\|_{L^p(\O;\g(\R_+;X))}\eqsim_{A,p,X,\theta}
\|G\|_{L^p(\O;\g(\R_+;H,X))}.
\end{align}

First let  $G:\R_+\times\O\to H\otimes \Dom(A^{\frac12 })$
be an adapted step process. By Proposition \ref{prop:Ito} applied to the UMD
space $\g(\R_+;X)$ and the $\g$-Fubini theorem (see the proof of Proposition \ref{prop:pindependence})
one has
\begin{equation}\label{eq:Itogamma}
\begin{aligned}
\ & \|D^{\theta} A^{\frac12-\theta} U\|_{L^p(\O;\g(\R_+;X))}
\\ & \quad \lesssim_{p,X}
\|t \mapsto [s\mapsto \one_{t>s} D_t^{\theta} A^{\frac12-\theta} S(t-s) G(s)]\|_{L^p(\O;\g(\R_+, dt;\g(\R_+,
ds;H,X)))},
\end{aligned}
\end{equation}
where $\lesssim$ can be replaced by $\eqsim$ if $p\in (1, \infty)$.

Pathwise we can estimate
\begin{align*}
\|t \mapsto [s\mapsto \one_{t>s} D_t^{\theta} & A^{\frac12-\theta} S(t-s) G(s)]\|_{\g(\R_+, dt;\g(\R_+,
ds;H,X))} \\ & \stackrel{\rm(a)}{=}
\|(i\lambda)^\theta A^{\frac12-\theta} e^{i\lambda s} (\lambda i + A)^{-1}
G(s)\|_{\g(\R, d\lambda;H,\g(\R_+, ds ;X))}
\\ & \stackrel{\rm(b)}{\eqsim} \|(i\lambda)^\theta A^{\frac12-\theta} e^{i\lambda
s} (\lambda i + A)^{-1} G(s)\|_{\g(\R_+\times\R, ds\times d\lambda ;H, X))}
\\ & \stackrel{\rm(c)}{=}
\|(i\lambda)^\theta A^{\frac12-\theta} (\lambda i + A)^{-1}
G(s)\|_{\g(\R_+\times\R;H ds\times d\lambda ;X))}
\\ & \stackrel{\rm(b)}{\eqsim} \|\lambda^\theta A^{\frac12-\theta} (\lambda i +
A)^{-1} G(s)\|_{\g(\R, d\lambda;\g(\R_+, ds ; H,X))}
\\ & \stackrel{\rm(d)}{=} \|z^{\frac12-\theta} A^{\frac12-\theta} (i + zA)^{-1}
G(s)\|_{\g(\R, d\frac{dz}{z} ;\g(\R_+, ds ;H,X))}
\\ & \stackrel{\rm(e)}{\eqsim} \|G\|_{\g(\R_+;H,X)}
\end{align*}
Here (a) follows by taking Fourier transforms and using \eqref{eq:Fisometry}, (b) follows from
\eqref{eq:isomalpha}, (c) follows from the right ideal
property and the identity $|i^{\theta} e^{is\lambda}|=1$, (d) follows by
simple rewriting and substitution $z = 1/\lambda$, and  (e) follows from Proposition \ref{prop:Hinftysquare} applied with
$\varphi(z)= z^{\frac12-\theta} (i +z)^{-1}$.
Combining the pathwise estimate with \eqref{eq:Itogamma} gives
\eqref{eq:gammaStochasticEst} for adapted step processes $G$. The general case follows from this by approximation.

(i): Stochastic maximal $\g$-regularity is obtained by taking $\theta=0$ in the above.
For adapted step processes $G$ with values in $H\otimes \Dom(A^{\frac12 })$,
the validity of the weak identity \eqref{eq:weaksol} is well known (cf.\ \cite{DPZ}).
The general case follows by approximation (cf. the proof of Theorem \ref{thm:gammaDeterministicEst}(i)).

(ii): First let $G:\R_+\times\O\to H\otimes \Dom(A^{\frac12 })$ be an adapted step process.
By \eqref{eq:gammaStochasticEst} applied with $\theta=0$ one sees that
$A^{\frac{1}{2}} U\in \g(\R_+;X)$ almost surely. Since $0\in \varrho(A)$, this implies that
$U\in \g(\R_+;\Dom(A^{\frac{1}{2}}))$ almost surely. This proves the result for $\theta=0$.
Moreover, $U\in \g(\R_+;\Dom(A^{\frac12-\theta}))$ for all $\theta\in
(0,\frac12)$
as well. Now the result follows from \eqref{eq:gammaDeterministicEstlem} and
Proposition \ref{prop:equivgammas}.

For general $G\in L_\F^p(\Omega;\g(\R_+;H,X))$ the result follows by approximation.

(iii): This follows follows from \cite[Theorem 4.2]{VW11}.
\end{proof}

\begin{corollary}\label{cor:maximregL^2}
Under the conditions of Theorem \ref{thm:gammaStochasticEst} one can replace (ii) by
\begin{enumerate}
\item[\rm(ii)$'$] (space-time regularity) For all $\theta\in [0,\frac12)$, $U\in
L^p(\O;H^{\theta, 2}(\R_+;\Dom(A^{\frac12-\theta})))$ and
\begin{align*}
\|U\|_{L^p(\O;H^{\theta,2}(\R_+;\Dom(A^{\frac12-\theta})))}\lesssim_{A,p,X,\theta}
\|G\|_{L^p(\O;\g(\R_+;H,X))},
\end{align*}
\end{enumerate}
\end{corollary}

\begin{remark}\label{rem:suffXA} If $X$ is a UMD Banach space and $A$
has a bounded $H^\infty$-calculus of angle $\pi/2$ and $0\in \varrho(A)$, then $A$ is $\g$-sectorial
by Proposition \ref{prop:functionalcalcgammasect}.
\end{remark}

\begin{remark}
The results of \cite{CoxVer2} imply that an upper estimate in \eqref{eq:NVW} still holds if the UMD property is replaced by the so-called decoupling property. Examples of Banach spaces with the decoupling property are the  UMD spaces and Banach spaces
isomorphic to a closed subspace of a space $L^1(\mu)$. One can check that Proposition \ref{prop:pindependence} and
Theorem \ref{thm:gammaStochasticEst} remain true for this class of spaces, the only difference being that
in Theorem \ref{thm:gammaStochasticEst} (ii) one cannot
replace $\lesssim$ by $\eqsim$ for $p\in (1, \infty)$.
\end{remark}

\section{Applications to (stochastic) evolution equations}\label{sec:see}

In this section we prove a $\gamma$-maximal regularity result for semilinear evolution
equations in a Banach space X of the form
\begin{equation}\tag{EE}\label{EE}
\left\{\begin{aligned}
U'(t)  +  A U(t)& = [F(t,U(t)) + f(t)], \qquad t\in
[0,T],\\
 U(0) & = u_0,
\end{aligned}
\right.
\end{equation}
and semilinear stochastic evolution
equations in $X$ of the form
\begin{equation}\tag{SEE}\label{SEE}
\left\{\begin{aligned}
dU(t)  +  A U(t)\, dt& = [F(t,U(t)) + f(t)] \,dt \\ & \qquad \qquad + [B(t,U(t)) + b(t)]\,dW_H(t), \qquad t\in
[0,T],\\
 U(0) & = u_0,
\end{aligned}
\right.
\end{equation}
where $A$ is $\g$-sectorial of angle $<\pi/2$ and $F$ and $G$ are nonlinearities satisfying suitable Lipschitz
and linear growth assumptions specified below. The initial value $u_0$ takes values in a suitable trace space
($X$ in the deterministic case, $X_\frac12$ in the stochastic case).

Evidently, \eqref{EE} is a special case of \eqref{SEE} by taking $B \equiv 0$ and $b\equiv 0$
and taking $u_0$ deterministic. For this reason we shall
discuss the stochastic case in detail, and leave the deterministic case as a simplification that the reader
may easily extract. In order to handle the stochastic term we shall always assume that $X$ be a UMD space,
but examination of the arguments shows that the deterministic case holds true for any Banach space $X$.

\subsection{Assumptions}

The assumptions are essentially the same as in  \cite{NVW11eq}, except that Lipschitz conditions
are now formulated in the corresponding $\g$-spaces.

\medskip\noindent
{\bf Hypothesis (H).}

\let\theenumi\ALTERWERTA
\let\labelenumi\ALTERWERTB
\let\ALTERWERTA\theenumi
\let\ALTERWERTB\labelenumi
\def\theenumi{{\rm (HA)}}
\def\labelenumi{(HA)}
\begin{enumerate}
\item\label{as:A}
There exists $w\in\R$ such that the operator $w+A$,
viewed as a densely defined operator on $X$ with domain $X_1:=\Dom(A)$,
has a bounded $H^\infty$-calculus on $X$ of angle
$0< \sigma<\frac12\pi$.
In what follows, for $\alpha\in (0,1)$ we write
$X_{\alpha} = [X,\Dom(A)]_{\alpha}$ for the complex interpolation space.
\let\theenumi\ALTERWERTA
\let\labelenumi\ALTERWERTB
\end{enumerate}

If \ref{as:A} holds for some $w\in\R$, then it holds for
any $w'>w.$ Furthermore, we may write $$-A + F = -(A +w') + (F + w'),$$
and note that a function $F$ satisfies the condition \ref{as:LipschitzF} below if and only if
$F+w'$ does. Thus, in what follows we may replace $A$ and $F$ by $A+w'$ and $F+w'$ and thereby
assume, without any loss of generality, that the operator $A$ is invertible.

Note that by Hypothesis \ref{as:A}, $X_{\alpha} = \Dom(A^{\alpha})$ for all $\alpha\in (0,1)$
(see \cite[Theorem 6.6.9]{Haase:2}).

\let\ALTERWERTA\theenumi
\let\ALTERWERTB\labelenumi
\def\theenumi{{\rm (HF)}}
\def\labelenumi{(HF)}
\begin{enumerate}
\item\label{as:LipschitzF}
The function $f:[0,T]\times\O\to X$ is adapted and strongly measurable and $f\in \g(0,T;X)$ almost surely.
The function $F:[0,T]\times\O\times X_1\to X$ is strongly
measurable and
\begin{enumerate}
\item for all $t\in [0,T]$ and $x\in X_1$ the random variable $\omega\mapsto F(t,\omega,
x)$ is strongly
$\F_t$-measurable;
\item there exist constants
$L_{F}$, $\tilde L_{F}$, $C_{F}$ such that for all
$\omega\in \O$, and $\phi_1, \phi_2\in \g(0,T;X_1)$,
\begin{align*}
\phantom{aaaaaaaaaaa}
\|F(\cdot,\omega, \phi_1) & - F(\cdot,\omega,\phi_2)\|_{\g(0,T;X)} \\ & \leq L_{F} \|\phi_1-\phi_2\|_{\g(0,T;X_1)} +
\tilde L_{F} \|\phi_1-\phi_2\|_{\g(0,T;X)}
\end{align*}
and
\begin{equation*}
\phantom{aaaaaaaa}
\|F(\cdot,\omega, \phi_1)\|_{\g(0,T;X)} \leq C_{F}(1+ \|\phi_1\|_{\g(0,T;X_1)}).
\end{equation*}
\end{enumerate}
\end{enumerate}
\let\theenumi\ALTERWERTA
\let\labelenumi\ALTERWERTB

\let\ALTERWERTA\theenumi
\let\ALTERWERTB\labelenumi
\def\theenumi{{\rm (HB)}}
\def\labelenumi{(HB)}
\begin{enumerate}
\item\label{as:LipschitzB}
The function $b:[0,T]\times\O\to \g(H,X_{\frac12})$ is adapted and strongly measurable
and $b\in \g(0,T;H,X_{\frac12})$ almost surely.
The function $B:[0,T]\times\O\times X_1\to \g(H,X_{\frac12})$
is strongly measurable and
\begin{enumerate}
\item for all $t\in [0,T]$ and $x\in X_1$ the random variable  $\omega\mapsto
B(t,\omega, x)$ is strongly $\F_t$-measurable;
\item there exist constants $L_{B}$, $\tilde L_{B}$, $C_{B}$ such that for all $t\in
[0,T]$, $\omega\in \O$, and $\phi_1,\phi_2\in\g(0,T;X_1)$,
\begin{align*}
\phantom{aaaaaaaaaaa}
\|B(\cdot,\omega, \phi_1) & - B(\cdot,\omega,\phi_2)\|_{\g(0,T;H,X_{\frac12})}
\\ & \leq L_{B} \|\phi_1-\phi_2\|_{\g(0,T;X_1)} +
\tilde L_{B} \|\phi_1-\phi_2\|_{\g(0,T;X)}
\end{align*}
and
\begin{equation*}
\phantom{aaaaaaaa}
\|B(\cdot,\omega, \phi_1)\|_{\g(0,T;H,X_{\frac12})} \leq C_{B}(1+ \|\phi_1\|_{\g(0,T;X_1)}).
\end{equation*}
\end{enumerate}
\end{enumerate}
\let\theenumi\ALTERWERTA
\let\labelenumi\ALTERWERTB

\let\ALTERWERTA\theenumi
\let\ALTERWERTB\labelenumi
\def\theenumi{{\rm (H$u_0$)}}
\def\labelenumi{(H$u_0$)}
\begin{enumerate}
\item \label{as:initial_value}
The initial value $u_0:\O\to X_{\frac12}$ is strongly $\F_0$-measurable.
\end{enumerate}
\let\theenumi\ALTERWERTA
\let\labelenumi\ALTERWERTB

The reader might have noticed that there is some redundancy in the conditions
\ref{as:LipschitzF} and \ref{as:LipschitzB} when we introduce the constants $L_F$ and $\tilde L_F$,
and $L_B$ and $\tilde L_B$, separately. The point here is that later on we shall impose a smallness condition
on the constants $L_F$ and $L_B$, but not on $\tilde L_F$ and $\tilde L_B$ which are allowed to be arbitrarily large.

\subsection{Solutions}

Throughout this subsection we assume that $X$ is a UMD Banach space and that {\rm (H)} is satisfied.
Observe that by Proposition \ref{prop:functionalcalcgammasect}, $w+A$ is $\g$-sectorial.

\begin{definition}\label{def:strongsol}
A process $U: [0,T]\times\Omega \to X$ is called a
{\em strong $\g$-solution} of \eqref{SEE} if it is strongly measurable and adapted,
and
\begin{enumerate}[(i)]
\item almost surely, $U\in \g(0,T;X_1)$;

\item for all $t\in [0,T]$, almost surely the following identity holds in $X$:
\begin{align*}
\phantom{aaaaaa}
U(t) + \int_0^t A U(s) \, ds = u_0 & + \int_0^t [F(s,U(s)) + f(s)] \, ds \\ & + \int_0^t
[B(s,U(s)) +b(s)] \, d W_H(s).
\end{align*}
\end{enumerate}
\end{definition}
Here the integrals are not Bochner integrals in general, but defined as in \eqref{eq:def-int}.
To see that the integrals are well defined, we note that, by \ref{as:A}, $A U\in \g(0,T;X)$ is strongly measurable and satisfies
\[\|A U\|_{\g(0,T;X)} \leq \|A\|_{\calL(X_1,X)} \|U\|_{\g(0,T;X_1)}\]
almost surely. Similarly, by \ref{as:LipschitzF} and
\ref{as:LipschitzB}, $F(\cdot,U(\cdot))$ and $ f$ belong to $\g(0,T;X)$ and $B(\cdot,U(\cdot))$ and $b$
belong to $\g(0,T;H,X_{\frac12})$ almost surely. The two deterministic integrals can now be
 interpreted almost surely in the sense of \eqref{eq:def-int}. For example, we interpret
\[\int_0^t A U(s) \, ds := (A U)(\one_{(0,t)}).\]
The stochastic integral is well defined in
$X_{\frac12}$ (and hence in $X$) by Proposition \ref{prop:Ito},
observing that $X_{\frac12}$ is a UMD space.

By Definition \ref{def:strongsol}, a strong
solution always has a version with continuous paths in $X$ such that,
almost surely, the identity in (ii) holds for all $t\in [0,T]$. Indeed, define
$\tilde{U}:[0,T]\times\O\to X$ by
\begin{align*}
\tilde{U}(t) := -\int_0^t A U(s) \, ds +  u_0 & + \int_0^t \big[F(s,U(s)) +f(s)\big]\, ds
+ \int_0^t \big[B(s,U(s)) + b(s) \big]\, d W_H(s),
\end{align*}
where we take continuous versions of the integrals on the right-hand side. From
the definitions of $U$ and $\tilde U$ one obtains, for all $t\in [0,T]$,
that $U(t) =
\tilde{U}(t)$ almost surely in $X$.
Therefore, almost surely, for all $t\in [0,T]$ one has
\begin{align*}
\tilde{U}(t) +\int_0^t A \tilde{U}(s) \, ds =  u_0 & + \int_0^t \big[F(s,\tilde{U}(s)) + f(s) \big]
\, ds + \int_0^t \big[B(s,\tilde{U}(s)) + b(s)\big] \, d W_H(s).
\end{align*}
From now on we choose this version whenever this is convenient.
We will actually prove much stronger regularity properties in Theorem
\ref{thm:SE} below.

\begin{definition}\label{def:mildsol}
A process $U: [0,T]\times\Omega \to X$ is called a
{\em mild $\g$-solution} of \eqref{SEE} if
it is strongly measurable and adapted, and
\begin{enumerate}[(i)]
\item almost surely, $U\in \g(0,T;X_1)$;

\item  for all $t\in [0,T]$, almost surely
the following identity holds in $X$:
\begin{align*}
U(t) = S(t) u_0 & + \int_0^t S(t-s)[ F(s,U(s)) + f(s)]\,ds
+ \int_0^t S(t-s)[B(s,U(s))+ b(s)]\,dW_H(s).
\end{align*}
\end{enumerate}
\end{definition}

The convolutions with $F(\cdot,U(\cdot))$ and $f$ are well defined as $X$-valued
processes by \ref{as:LipschitzF}. The stochastic convolutions with $B(\cdot,U(\cdot))$ and $b$ are
well defined as $X_{\frac12}$-valued processes
(and hence as an $X$-valued process) by \ref{as:LipschitzB}, the fact that $X_{\frac12}$
is a UMD space and Proposition \ref{prop:Ito}.

The following type of result is well known and the proofs extend to our situation (cf. \cite{DPZ, NVW11eq}).

\begin{proposition}\label{prop:strongmild}
Let $X$ be a UMD Banach space and let {\rm (H)} be satisfied.
A process $U:[0,T]\times\Omega\to X$ is a strong solution
of \eqref{SEE} if and only if it is a mild solution of \eqref{SEE}.
\end{proposition}

\subsection{Well-posedness}

The main result of this section is the following maximal $\g$-regularity result.

\begin{theorem}\label{thm:SE}
Let $X$ be a UMD Banach space with property $(\alpha)$ and
let {\rm (H)} be satisfied. Let $p\in (0,\infty)$ be given and assume that $f\in L^p_{\F}(\O;\g(0,T;X))$
and $b\in L^p_{\F}(\O;\g(0,T;H,X_{\frac12}))$.
There exists  a constant $\delta>0$, depending only on $A$, $p$, $T$, $X$,
such that if the Lipschitz constants $L_F$ and $L_B$
satisfy $\max\{L_F,L_B\}<\delta$,
then the following assertions hold:
\begin{enumerate}
\item[\rm(i)]
The problem \eqref{SEE}
has a unique strong $\g$-solution $U\in L^0_{\F}(\O;\g(0,T;X_1))$.
Moreover, $U$ has a version with trajectories in $C([0,T];X_{\frac12})$.

\item[\rm(ii)] If $u_0\in L_{\F_0}^p(\Omega; X_{\frac12} )$,
then the strong solution $U$ given by part {\rm (i)}
belongs to the space $L^p_{\F}(\O;\g(0,T;X_1))\cap
L_{\F}^p(\O;C([0,T];X_{\frac12}))$  and satisfies
\begin{align*}
\|U\|_{L^p(\O;\g(0,T;X_1))} &\leq
C(1+\|u_0\|_{L^p(\O;X_{\frac12} )}),\\
\phantom{aaaa} \|U\|_{L^p(\O;C([0,T];X_{\frac12} ))} & \leq
C(1+\|u_0\|_{L^p(\O;X_{\frac12})}),
\end{align*}
with constants $C$ independent of $u_0$.

\item[\rm(iii)] For all $u_0, v_0\in L_{\F_0}^p(\O;X_{\frac12})$,
the corresponding strong solutions $U,V$ satisfy
\begin{align*}
\|U-V\|_{L^p_{\F}(\O;\g(0,T;X_1))} &\leq C
\|u_0-v_0\|_{L^p(\O;X_{\frac12})},\\
\phantom{aaaa} \|U-V\|_{L^p(\O;C([0,T];X_{\frac12}))} & \leq C\|u_0-
v_0\|_{L^p(\O;X_{\frac12})},
\end{align*}
with constants $C$ independent of $u_0$ and $v_0$.
\end{enumerate}
\end{theorem}

\begin{proof}
A proof is obtained by repeating the proof of the corresponding maximal $L^p$-regularity
result of \cite{NVW11eq} {\em verbatim}. Here the trace space
$D_A(1-\frac1p,p)$ of \cite{NVW11eq} is should be replaced by the trace space $X_{\frac12}$.
Moreover, the fixed point spaces used in the proof \cite{NVW11eq} should be replaced by
\begin{align*}
Z_{\theta,\kappa} & = L^p_{\F}(\O;\g(0,\kappa;X_\theta)),\\
Z^H_{\theta,\kappa} & = L^p_{\F}(\O;\g(0,\kappa;H,X_{\theta})).
\end{align*}
where $\kappa \in (0,T]$ and $\theta\in [0,1]$.
The proof gives the following explicit smallness condition on the Lipschitz coefficients.
First rescale $A$ to $A+w$, where $w\in \R$ is large enough in order that the spectrum of $A+w$
is contained in the open right half-plane, and write $S_w(t) = e^{-wt}S(t)$.
Denote by $K_p^*$ the norm of the operator
$g\mapsto S_w *g$ from
$L^p_{\F}(\O;\g(\R_+;X))$ into $L^p_{\F}(\O;\g(\R_+;X_1))$, and by
$K_p^\diamond$ the norm of the operator $G\mapsto S_w\diamond G$ from
$L^p_{\F}(\O;\g(\R_+;H,X_{\frac12}))$
into $L^p_{\F}(\O;\g(\R_+;X_1))$.
Then the conclusions of the theorem hold if $L_F K_p^* + L_B K_p^\diamond < 1$.
\end{proof}

\begin{remark}
Applying Theorem \ref{thm:gammaDeterministicEst} (ii) to the space $X$ and
Theorem \ref{thm:gammaStochasticEst}(ii) to the space $X_{\frac12}$
one can prove in the same way that
\[U\in L^0(\O;\g^{\theta}(0,T;X_{1-\theta})) \ \text{for all $\theta\in
[0, \tfrac12)$}\]
and the following estimates hold:
\begin{align*}
\|U\|_{L^p(\O;\g^{\theta}(0,T;X_{1-\theta}))} &\leq
C(1+\|u_0\|_{L^p(\O;X_{\frac12})}),
\\ \|U-V\|_{L^p(\O;\g^{\theta}(0,T;X_{1-\theta}))} &\leq C
\|u_0-v_0\|_{L^p(\O;X_{\frac12})},
\end{align*}
where $U$ and $V$ are the solutions with initial values $u_0$ and $v_0$
respectively.
If $X$ has cotype $q\in [2, \infty)$, then by Proposition \ref{prop:gammsobolev},
\[U\in L^p(\O;B^{\theta+\frac{1}{q}-\frac12}_{q,q}([0,T];X_{1-\theta})) \ \text{for all
$\theta\in [0, \tfrac12)$}.\]
By Remark \ref{rem:improvedembedding}, one can replace the Besov scale by the Bessel-potential scale
if $q=2$ or $X$ is a $q$-concave Banach lattice.
\end{remark}

\begin{remark}
The smallness condition cannot be omitted in Theorem \ref{thm:SE}. A detailed
discussion in the $L^p$-maximal regularity setting on this matter can be found in \cite{BrzVer11}.
For $p=2$ and $X$ a Hilbert space, this discussion applies to the present setting as well.
 See also \cite{KimLee} for a related result for systems.
\end{remark}

\begin{remark}
Inspection of the the proof, in combination with Remark \ref{rem:suffXA}, reveals
that the results of Theorem \ref{thm:SE}
still hold for Banach spaces $X$ which have the decoupling property and property $(\alpha)$.
In particular, this includes the case $X = L^1(\mu)$.
\end{remark}

For the convenience of the reader, we also include an explicit formulation of the
corresponding result for the deterministic problem \eqref{EE}. We take $B \equiv 0$, $b\equiv 0$,
and assume that the initial value $u_0$ is a fixed element of $X$. Hypothesis H$_{\rm det}$ is now
understood to be the same as (H), with the following modifications:
\begin{enumerate}
 \item[\rm(i)]  all objects are taken to be deterministic;
 \item[\rm(ii)]  assumption (HB) is canceled.
\end{enumerate}

\begin{theorem}\label{thm:EE}
Let $X$ be Banach space with finite cotype,
let {\rm H$_{\rm det}$} be satisfied and assume in addition
that some translate of $A$ is $\g$-sectorial of angle $<\pi/2$. Let $p\in (0,\infty)$ be given.
There exists a constant $\delta>0$, depending only on $A$, $p$, $T$, $X$,
such that if $L_F<\delta$,
then the following assertions hold:
\begin{enumerate}
\item[\rm(i)]
For all $u_0\in X_{\frac12}$, the problem \eqref{EE}
has a unique strong $\g$-solution $U$. It belongs to $\g(0,T;X_1)\cap \g^1(0,T;X)$ and satisfies
\begin{align*}
\|U\|_{\g(0,T;X_1)} + \|U\|_{\g^1(0,T;X)}&\leq
C(1+\|u_0\|_{ X_{\frac12}}),
\end{align*}
and for all $\theta\in [0,\frac12)$  one has $u\in C^{\theta}([0,T];X_{1-\theta})$ and
\[\|U\|_{C^{\theta}([0,T];X_{1-\theta})}  \leq C(1+\|u_0\|_{ X_{\frac12}}),\]
with constants $C$ independent of $u_0$.

\item[\rm(ii)] For all $u_0, v_0\in X_{\frac12}$,
the corresponding strong solutions $U,V$ satisfy
\begin{align*}
\|U-V\|_{\g(0,T;X_1)} + \|U-V\|_{\g^1(0,T;X)}&\leq C
\|u_0-v_0\|_{ X_{\frac12}},\\
\phantom{aaaa} \|U-V\|_{C^{\theta}([0,T];X_{1-\theta})} & \leq C\|u_0-
v_0\|_{ X_{\frac12}}, \ \ \   \theta\in [0,\tfrac12 ),
\end{align*}
with constants $C$ independent of $u_0$ and $v_0$.
\end{enumerate}
\end{theorem}
The space $X$ need not be UMD; the UMD property comes in only when dealing with
stochastic integrals.
We do need a finite cotype assumption to ensure that $S u_0\in \g(\R_+;X_1)$
(by the second part of Proposition \ref{prop:Hinftysquare}).

The $\g$-sectoriality condition is automatically fulfilled if (H) holds and
$X$ has property $(\Delta)$ (see Proposition \ref{prop:functionalcalcgammasect}).

\section{Time-dependent case}\label{sec:see-time}

In the same setting as before we now consider the following time-dependent version of \eqref{SEE} with an
operator family $A = (A(t))_{t\in [0,T]}$ consisting of densely defined operators on $X$
with common domains $\Dom(A(t)) =: X_1$:
\begin{equation}\tag{SEE$'$}\label{SEE'}
\left\{\begin{aligned}
dU(t)  +  A(t) U(t)\, dt & =  [F(t,U(t)) + f(t)]\,dt \\ & \qquad  +
[B(t,U(t)) +b(t)]\,dW_H(t), \quad t\in
[0,T],\\
 U(0) & = u_0.
\end{aligned}
\right.
\end{equation}
Below we shall extend the definition of a strong solution
to the time-dependent problem \eqref{SEE'} for operators
$A$ and prove the existence and uniqueness
of strong solutions for \eqref{SEE'} by means of maximal regularity techniques.

Throughout this section we replace Hypothesis \ref{as:A}
by the following hypothesis \ref{as:A'} and
we say that {\em Hypothesis {\rm (H)}$'$ holds} if \ref{as:A'},
\ref{as:LipschitzF}, \ref{as:LipschitzB}, and \ref{as:initial_value} hold,
with

\let\ALTERWERTA\theenumi
\let\ALTERWERTB\labelenumi
\def\theenumi{{\rm (HA)$'$}}
\def\labelenumi{(HA)$'$}
\begin{enumerate}
\item\label{as:A'}
Each operator $A(t)$, viewed as a densely defined operator
on $X$ with domain $X_1$, is invertible and
has a bounded $H^\infty(\Sigma_\sigma)$-calculus, with $\sigma\in (0,\frac12\pi)$ independent of $t\in [0,T]$.
There is a constant $C$, independent of $t\in [0,T]$,
such that for all $\varphi\in H^\infty(\Sigma_{\sigma})$,
\[\|\varphi(A(t))\|\leq C \|\varphi\|_{H^\infty(\Sigma_{\sigma})}.\]
The Banach space $X$ has type $p_0\in (1, 2]$ and cotype $q_0\in [2, \infty)$,
and we have $A\in B^{\frac{1}{r}}_{r,1}([0,T]; \calL(X_1,X))$
for some $r\in [1, \infty]$ satisfying $\frac1r \geq \frac{1}{p_0} - \frac{1}{q_0}$.
\end{enumerate}
\let\theenumi\ALTERWERTA
\let\labelenumi\ALTERWERTB
The first part of Hypothesis \ref{as:A'} implies that the operators $-A(t)$
generate bounded analytic $C_0$-semigroups on $X$ for which the usual sectoriality
estimate holds holds uniformly in $t\in [0,T]$.

Assumption \ref{as:A'} together
with \cite[Theorem 5.1]{HytVer} implies that $\{A(t): t\in [0,T]\}\subseteq \calL(X_1,X)$ is $\g$-bounded.
In the next lemma we show that the variation of the
$\g$-bounds becomes arbitrary small on small intervals.
\begin{lemma}\label{lem:Aphi}
Let {\rm (H)$'$} be satisfied. For all $\varepsilon>0$ there exists a $\d>0$ such that
for all $0\le s\le s'\le T$ with $s'-s\le \d$ we have
\[\g(\{A(u)-A(v): u,v\in [s,s']\})<\varepsilon.\]
\end{lemma}
\begin{proof}The proof uses some standard facts about (vector-valued) Besov spaces, for which we
refer the reader to \cite{BeLo, Tr1}.

By standard real interpolation arguments (see \cite[Theorem 4.2.2]{Tr1}) we can find an extension
$\Phi\in B^{\frac{1}{r}}_{r,1}(\R; \calL(X_1,X))$ of $A$.
Let $(\varphi_m)_{m\geq 0}$ be the usual Littlewood-Paley
decomposition and let $\Phi_{m} = \varphi_m*\Phi$. Then
$\Phi_{mn}=0$ if $|n-m|>1$. Let further $\varphi_{-1} = 0$.

It follows from the proof of \cite[Theorem 5.1]{HytVer}
that for all $m,n\geq 0$,
\[\g(\varphi_m*\Phi_n(u): u\in [s, s'])\leq C 2^{m/r} \|\Phi_n\|_{L^r(\R;\calL(X_1, X))}.\]
Therefore, writing $\Phi_{mn} = \varphi_m*\Phi_n$, by the identity
$$\sum_{n=0}^\infty \sum_{m=n-1}^{n+1} \varphi_m*\Phi_{mn} = \Phi$$
and \cite[Lemma 2.4]{We} we find that for every $N\geq 0$,
\begin{align*}
\g(A(u)-A(v):u,v\in [s, s'])
&\leq \sum_{n=0}^\infty \sum_{m=n-1}^{n+1} \g(\Phi_{mn}(u)-\Phi_{mn}(v): u,v\in [s, s'])
\\ & \leq \sum_{n=0}^N \sum_{m=n-1}^{n+1} \g(\Phi_{mn}(u)-\Phi_{mn}(v): u,v\in [s, s'])
\\ & \qquad + 2\sum_{n=N+1}^\infty \sum_{m=n-1}^{n+1} \g(\Phi_{mn}(w): w\in [s, s'])
\\ & \leq \sum_{n=0}^N \sum_{m=n-1}^{n+1} \g(\Phi_{mn}(u)-\Phi_{mn}(v): u,v\in [s, s'])
\\ & \qquad + C \sum_{n=N+1}^\infty 2^{n/r} \|\Phi_n\|_{L^r(\R;\calL(X_1, X)}
\end{align*}
for a suitable constant $C$ independent of $N$ and $\Phi$.
Let $\varepsilon>0$ be arbitrary. By the equivalence of norms
$$\|\Phi\|_{B^{\frac{1}{r}}_{r,1}(\R; \calL(X_1,X))} \eqsim \sum_{n\geq0} 2^{n/r} \|\Phi_n\|_{L^r(\R;\calL(X_1,X))}$$
we may fix $N\geq 0$ so large that
\[\sum_{n\geq N+1} 2^{n/r} \|\Phi_n\|_{L^r(\R;\calL(X_1,X))}<\varepsilon/(2C).\]
Fix $0\leq n\leq N$ and $m\in \{n-1, n, n+1\}$.
Note that $\Phi_{mn}\in W^{1,1}(0,T;\calL(X_1, X))$. In particular, there exists
a number $\delta_{mn}>0$ such that $\|\Phi_{mn}'\|_{L^1(s, s';\calL(X_1, X))}<\varepsilon/(6N)$
whenever $|s-s'|<\delta_{mn}$ and $s,s'\in [0,T]$. Let
$\delta = \min\{\delta_{mn}: \ 0\leq n\leq N, \
n-1\leq m\leq n+1\}$.
We claim that $\g(\Phi_{mn}(u)-\Phi_{mn}(v): u,v\in [s, s'])<\varepsilon/(6N)$ whenever $|s-s'|<\delta$.
To prove this it suffices to consider pairs $(u,v)$ with $u\leq v$.
Choose arbitrary $x_1, \ldots, x_k\in X$ and $(u_i)_{i=1}^k, (v_i)_{i=1}^k\in [s', s]$ with $u_i<v_i$ for
every $i\leq k$. By the triangle inequality and Kahane's contraction principle,
\begin{align*}
\Big\|\sum_{i=1}^k \g_i (\Phi_{mn}(u_i)-\Phi_{mn}(v_i))x_i\Big\|_{L^2(\O;X)}
&= \Big\|\int_{s'}^s \sum_{i=1}^k \g_i \one_{[u_i, v_i]}(t) \Phi_{mn}'(t) x_i\,dt \Big\|_{L^2(\O;X)}
\\ & \leq \int_{s'}^s  \Big\|\sum_{i=1}^k \g_i  \one_{[u_i, v_i]}(t)
\Phi_{mn}'(t) x_i\Big\|_{L^2(\O;X)}\,dt
\\ & \leq \int_{s'}^s  \Big\|\sum_{i=1}^k \g_i
\Phi_{mn}'(t) x_i\Big\|_{L^2(\O;X)}\,dt\\ & \leq \int_{s'}^s  \|\Phi_{mn}'(t)\| \,dt \ \cdot \  \Big\|\sum_{i=1}^k \g_i  x_i\Big\|_{L^2(\O;X)}
\\ & \leq \frac{\varepsilon}{6N} \Big\|\sum_{i=1}^k \g_i  x_i\Big\|_{L^2(\O;X)}.
\end{align*}
This proves the claim. Combination of the assertions yields that
\begin{align*}
\g(A(u)-A(v):u,v\in [s, s'])  \leq 3N\frac{\varepsilon}{6N} + C\frac{\varepsilon}{2C} = \varepsilon.
\end{align*}
\end{proof}

\begin{definition}
Let $X$ be a UMD space and let {\rm (H)$'$} be satisfied.
A process $U: [0,T]\times\Omega \to X$ is called a
{\em strong solution} of \eqref{SEE'} if it is strongly measurable and adapted,
and
\begin{enumerate}
\item[\rm(i)] almost surely, $U\in \g(0,T;X_1)$;
\item[\rm(ii)] for all $t\in [0,T]$, almost surely the following identity holds in $X$:
\begin{equation}
\label{strongsol'}\begin{aligned}
U(t) + \int_0^t A(s) U(s) \, ds  = u_0 & +  \int_0^t [F(s,U(s)) + f(s)] \, ds
 + \int_0^t [B(s,U(s)) + b(s)]\, d W_H(s).
\end{aligned}
\end{equation}
\end{enumerate}
\end{definition}

As before, under {\rm (H)$'$} all integrals are well defined. Note that $A U\in
\g(0,T;X)$ almost surely by Lemma \ref{lem:Aphi}.
Again $U$ has a pathwise continuous version
for which, almost surely, the identity in (ii) holds for all $t\in [0,T]$.

\begin{theorem}\label{thm:SE2}
Let $X$ be a UMD space with property $(\a)$ and let {\rm (H)$'$} be satisfied.
There exists a constant $\delta>0$ such that if
the Lipschitz constants $L_F$ and $L_B$
satisfy $\max\{L_F, L_B\} <\delta$,
then the assertions of Theorem \ref{thm:SE} (i),
(ii) and (iii) remain true for the problem \eqref{SEE'}.
\end{theorem}

\begin{proof}
Using Lemma \ref{lem:Aphi} we find a partition $0=s_0<s_1< \ldots< s_M=T$ such that for $m=1,
\ldots, M$ one has
\[\g(\{A(u)-A(v): u,v\in [s_{m-1},s_m]\})< \theta/2.\]
By Proposition \ref{prop:KW}, for all $m=1, \ldots, M$ and $\phi\in \g(s_{m-1},s_{m};X_1)$
one then has
\[\|(A - A(s_{m-1})) \phi\|_{\g(a,b;X)}\leq \tfrac12{\theta} \|\phi\|_{\g(s_{m-1},s_m;X_1)}.\]
Without loss of generality
we may replace $A$ by $A-w$ so as to achieve that for $m=1,\ldots,M$ on has
$K^*_{p,m} L_F  + K^\diamond_{p,m} L_B<1$, say $K^*_{p,m}  L_F  + K^\diamond_{p,m}
L_B = 1-\theta$ with $\theta\in (0,1)$. Here $K^*_{p,m}$ and $K^\diamond_{p,m}$ are the
norms associated with the operators $A(s_{m-1})$ as before.

We first solve the problem \eqref{SEE'} on the interval $[s_0, s_1]$. Let $F_{A,0}:[s_0,s_1]\times\O\times
X_1\to X$ be defined by $F_{A,0}(t,x) = F(t,x) - A(t)x + A(0)x$. Then $F_{A,0}$
satisfies \ref{as:LipschitzF} (with $F$ replaced by $F_{A,0}$). Moreover,
$L_{F_{A,0}}\leq L_F + \frac12\theta$ and $\tilde L_{F_{A,0}}\leq
\tilde L_{F}$,
and therefore the condition of Theorem \ref{thm:SE}
holds for the equation with $F$ replaced by $F_{A,0}$ and $A$ replaced by $A(0)$,
with constants satisfying $K^*_{p,0} L_{F_{A,0}}  + K^\diamond_{p,0} L_B \le 1-\frac12\theta$.
Hence
Theorem \ref{thm:SE} implies the existence of a unique strong solution $U\in
L^0_{\F}(\O; L^p(0,s_1;X_1))$. Then
almost surely, for all $t\in [0,s_1]$ the following identity holds in $X$:
\begin{align*}
U(t) + \int_0^t A(0) U(s) \, ds = u_0& + \int_0^t F_{A,0}(s,U(s)) + f(s) \, ds
+ \int_0^t
B(s,U(s)) + b(s) \, d W_H(s)
\end{align*}
and \eqref{strongsol'} holds on $[0,s_1]$ almost surely.
Moreover, the assertions of Theorem \ref{thm:SE} (i), (ii) and (iii) hold on
$[0,s_1]$.

Now we proceed inductively. Suppose we know that the assertions
of Theorem \ref{thm:SE} (i), (ii) and (iii) hold for the problem \eqref{SEE'}
on the interval $[0, s_{m}]$ with $m\leq M$. If $m=M$, there is
nothing left to prove. If $m<M$, we shall prove next existence and uniqueness on
the interval $[s_{m},s_{m+1}]$.

Consider the problem
\begin{equation}\label{eq:VFAn}
\left\{\begin{aligned}
dV(t)  +  A(s_{m}) V(t)\, dt& = [F_{A,{m}}(t,V(t)) + f(t)] \,dt \\ & \qquad
 + [B(t,V(t))+ b(t)]\,dW_H(t),\quad t\in [s_{m}, s_{m+1}],\\
 V(s_{m}) & = U(s_{m})
\end{aligned}
\right.
\end{equation}
with $F_{A,m} = F(t,x) - A(t) + A(s_{m})$. As before, Theorem \ref{thm:SE} can
be applied to obtain a unique strong solution $V\in L^0_{\F}(\O;
L^p(s_{m},s_{m+1};X_1))$ and assertions (i), (ii) and (iii) of Theorem
\ref{thm:SE} hold
for the solution $V$ of \eqref{eq:VFAn}. Now we extend $U$ to $[0,s_{m+1}]$ by
setting $U(t) := V(t)$ for $t\in [s_m,s_{m+1}]$. Then $U$ is in $L^0_{\F}(\O;
\g(0,s_{m+1};X_1))$
 and has a version with trajectories in $C([0,s_{m+1}];X_{\frac12})$.
Moreover, using the induction hypothesis, one sees that it is a
strong solution on $[0,s_{m+1}]$. It is also the unique strong solution on
$[0,s_{m+1}]$. Indeed, let $W\in L^0_{\F}(\O; \g(0,s_{m+1};X_1))$ be another
strong solution on $[0,s_{m+1}]$. By the induction hypothesis we
have $W = U$ in $L^0_{\F}(\O;\g(0,s_{m};X_1))$. In particular, the definition
of a
strong solution implies that $W(s_{m}) = U(s_m)$ almost surely. Now one can see
that $W$ is strong solution of \eqref{eq:VFAn} on $[s_{m},s_{m+1}]$. Since the
solution of \eqref{eq:VFAn} is unique, it follows that also $W = V$ in
$L^0_{\F}(\O;\g(s_{m},s_{m+1};X_1))$. Therefore, the definition of $U$ shows
that $U = W$ in $L^0_{\F}(\O;\g(0,s_{m+1};X_1))$. The other results in (i),
(ii)
and (iii) for $U$ on $[0,s_{m+1}]$ follow from the corresponding results for $V$
as well. This completes the induction step and the proof.
\end{proof}

\begin{remark}
The H\"older continuity assumption on $A$ can be weakened a bit; for instance,
only piecewise H\"older continuity would suffice. The main ingredient in the above
approach is that the range of $A$ is $\g$-bounded in $\calL(X_1, X)$ and for
each $\varepsilon>0$ there is a dense collection of $t\in [0,T]$ for which
\[\limsup_{\delta\downarrow 0} \g\{A(t+h)- A(t): h\leq \delta\} <\varepsilon.\]
If $X$ is a Hilbert space, the assumption reduces to piecewise continuity of
$A$.
\end{remark}

\begin{remark} The usage of constants $K_{p,m}$ depending on $m$ in the above proof can be avoided by
observing that they can be uniformly bounded by a constant depending only upon $p$,
$X$ and the uniform $H^\infty$-constant of the operators $A(t)$. This has already
been implicitly used in the proof of \cite[Theorem 5.2]{NVW11eq}.
\end{remark}

\section{Application to a stochastic heat equation with gradient noise\label{sec:appl}}

As an application we show how one can solve a stochastic heat equation with
gradient noise in an $L^q(\R^d)$-space, where $q\in (1, \infty)$. For $q\in [2,
\infty)$, the assertions are different from those in \cite{Kry} and
\cite{NVW11eq}.

On $\R^d$ we consider the second order SPDE
\begin{equation}\label{eq:heateq}
 \left\{
   \begin{array}{ll}
     du & = \Delta u + F(u) + B(u) \, d W_H, \\
     u(0,\cdot) & = u_0.
   \end{array}
 \right.
\end{equation}
Let $s\in \R$ be fixed. The realization of the Laplace operator $\Delta$ on $H^{s,q}(\R^d)$, also denoted by $\Delta$,
has domain $H^{s+2,q}(\R^d)$ and has a bounded $H^\infty$-calculus of angle $<\pi/2$.
We recall that for any Hilbert space $H$ and any $\sigma\in\R$ and $p\in (0,\infty)$
we have a natural isomorphism of Banach spaces
$$ \g(H;H^{\sigma,q}(\R^d)) \simeq H^{\sigma,q}(\R^d;H).$$
This allows us to formulate our results without any explicit reference
to $\g$-norms.

We shall assume that $F:H^{s+2, q}(\R^d)\to H^{s, q}(\R^d)$ and $B:H^{s+2, q}(\R^d)\to
H^{s+1,q}(\R^d;H)$ are functions such that for all $\phi_1, \phi_2\in
H^{s+2,q}(\R^d;L^2(0,T))$ one has
\begin{equation}\label{eq:exF}
\begin{aligned}
 & \|F(\phi_1) - F(\phi_2)\|_{H^{s,q}(\R^d;L^2(0,T))}
\\ & \qquad \leq L_{F} \|\phi_1 -
\phi_2\|_{H^{s+2,q}(\R^d;L^2(0,T))} + \tilde{L}_{F} \|\phi_1
-\phi_2\|_{H^{s,q}(\R^d;L^2(0,T))},
\end{aligned}\end{equation}
\begin{equation}\label{eq:exB}
\begin{aligned}
 & \|B(\phi_1) - B(\phi_2)\|_{H^{s+1,q}(\R^d;L^2(0,T))}
\\ & \qquad \leq L_{B} \|\phi_1 -
\phi_2\|_{H^{s+2,q}(\R^d;L^2(0,T))} + \tilde{L}_{B} \|\phi_1
-\phi_2\|_{H^{s+1,q}(\R^d;L^2(0,T))}.
\end{aligned}
\end{equation}
If the Lipschitz constants $L_{F}$ and $L_{B}$ are small enough, then for every $q\in (1, \infty)$ and
every $u_0\in L^0(\O;\F_0,H^{s+1,q}(\R^d))$, \eqref{eq:heateq} has a unique
solution
\[u\in L^0(\O;H^{s+2,q}(\R^d;L^2(0,T)))\cap L^0(\O;C([0,T];H^{s+1,q}(\R^d))).\]
This follows from Theorem \ref{thm:SE} with $X = H^{s,q}(\R^d)$, $X_1 =
H^{s+2,q}(\R^d)$.

Let us now consider the case $s=-1$ in more detail.
The assertion $u\in L^0(\O;H^{1,q}(\R^d;L^2(0,T)))$ can be restated as
\[\int_{\R^d}  \Big(\int_0^T |D u(t,x)|^2\, dt\Big)^{q/2} \, dx<\infty \quad
\text{almost surely}.\]

Taking $H = \ell^2$ with orthonormal basis $(h_n)$ and taking $w_n = W_H h_n$, one could
consider noise of the form
\[B(u) \,dW_H = \sum_{n\geq 1} g_n(u, Du)\, d w_n,\]
where
\begin{align*}
\Big(\sum_{n\geq 1}  |g_n(x,a) - g_n(y, b)|^2\Big)^{\frac{1}{2}} \leq L_{g,1}|x-y| +
L_{g,2} |a-b|,  \ x, y\in \R, a,b\in \R^d,
\end{align*}
with $x, y\in \R$, $a,b\in \R^d$, and with $L_{g,2}$ sufficiently small. Indeed,
\eqref{eq:exB} follows from the next lemma.

\begin{lemma}\label{lem:glipschitz}
Let $\OO\subseteq \R^d$ be an open set and let $f:\R\times\R^d\times\R^{d\times d}\to \R$ be a
Lipschitz function with
\begin{align*}
|&f(x, a, A) - f(y, b, B)|  \leq L_{f,1} |x-y| + L_{f,2} |a-b| + L_{f,3} |A-B|,
\end{align*}
where $x, y\in \R, a,b\in \R^d, A,B\in \R^{d\times d}$. Let $p\in [1, \infty)$.
Then for all $\phi_1, \phi_2\in \g(0,T;W^{2,p}(\mathcal{O}))$,
\begin{align*}
\| f(\phi_1,  D\phi_1, D^2 \phi_1) & - f(\phi_2,  D\phi_2, D^2\phi_2)\|_{\g(0,T;L^p(\OO))}
\\ & \leq C L_{f,1} \|\phi_1 - \phi_2\|_{\g(0,T;L^p(\OO))}
+ C L_{f,2} \|D \phi_1 - D\phi_2\|_{\g(0,T;L^p(\OO;\R^d))}
\\ & \qquad + C L_{f,3} \|D^2\phi_1 - D^2\phi_2\|_{\g(0,T;L^p(\OO;\R^{d\times
d}))}.
\end{align*}
\end{lemma}
\begin{proof}
By \eqref{eq:gammaBfunc} we have
\begin{align*}
\|f(\phi_1, & D\phi_1, D^2 \phi_1) - f(\phi_2, D\phi_2, D^2
\phi_2)\|_{\g(0,T;L^p(\OO))}
\\ & \eqsim_{p}  \|f(\phi_1,  D\phi_1, D^2 \phi_1) - f(\phi_2, D\phi_2, D^2
\phi_2)\|_{L^p(\OO;L^2(0,T))}
\\ & \leq \Big\|L_{f,1} |\phi_1 - \phi_2| + L_{f,2} \|D \phi_1 - D\phi_2\| +
L_{f,3} |D^2\phi_1 - D^2\phi_2|\Big\|_{L^p(\OO;L^2(0,T))} \\ & \leq L_{f,1}
\|\phi_1 - \phi_2\|_{L^p(\OO;L^2(0,T))} + L_{f,2} \|D \phi_1 -
D\phi_2\|_{L^p(\OO;L^2(0,T;\R^d))} \\ & \qquad \qquad +L_{f,3} \|D^2\phi_1 -
D^2\phi_2\|_{L^p(\OO;L^2(0,T;\R^{d\times d}))}.
\end{align*}
Now the result follows from another application of \eqref{eq:gammaBfunc}.
\end{proof}

If $s=0$, one can allow nonlinearities of the form
\[F(v) = f(v, Dv,D^2v),\] where
\begin{align*}
|&f(x, a,M) - f(y, b,N)|  \leq L_{f,1} |x-y| + L_{f,2} |a-b|+L_{f,3} |M-N|,
\end{align*}
for $x, y\in \R$ and $M,N\in \R^{d\times d}$, with $L_{f,3}$
sufficiently small. Indeed, \eqref{eq:exF} follows from Lemma
\ref{lem:glipschitz}.

\begin{remark}
This example can be extended to general second order or higher order operators elliptic operators as in \cite{DuSi};
appropriate changes to the nonlinearities $F$ and $B$ should be made. Using
Theorem \ref{thm:SE2}, one may also allow time-dependent operators $A(t)$.
\end{remark}

\section{Comparison}

The theory presented here provides an alternative approach to
the theory of maximal $L^p$-regularity (as presented in
\cite{KuWe, We} and the references therein) and stochastic
maximal $L^p$-regularity (developed recently in \cite{NVW10, NVW11eq}).
A detailed comparison of the latter with known stochastic
maximal regularity results in the literature
(such as in, e.g., \cite{Brz1, DPL, DPZ, DeschLonden, BrzHau, KimKH08, Kry, MiRo04}.)
have been given in \cite{NVW10, NVW11eq}. We also mention the papers \cite{Hofmanova, JentzenRoeckner}, where
higher order regularity in the space variables is obtained under additional structural assumptions on the
nonlinearities.

The main differences between the approach presented here and that in \cite{NVW10, NVW11eq}
are the replacement of the Bochner norms by $\g$-norms
and the replacement of the trace space $X_{1-\frac1p,p}$ (with $p>2$) by $X$ in the deterministic case and
by $X_{\frac{1}{2}}$ in the stochastic case.
Thus, the theory presented here allows rougher initial values, but the price to pay is
that pathwise solutions are obtained in $\g(0,T;X_1)$ instead of $L^p(0,T;X_1)$.
A further difference is that we can handle more general Banach spaces and that,
in the stochastic case, we obtain estimates for the moments of all orders $0<p<\infty$
instead of only for $2 < p<\infty$ ($2\le p<\infty$ in case $X$ is a Hilbert space).

In the deterministic case (Theorem \ref{thm:EE}) we only needed to assume that $X$ has finite cotype;
in contrast, in the results of \cite{KuWe, We} the space $X$ is
assumed to be UMD. In the stochastic case we can allow UMD Banach spaces $X$ with property $(\a)$
(this includes all spaces
isomorphic to a closed subspace of $L^q(\mu)$ with $q\in (1,\infty)$), while
the results of \cite{NVW11eq} could (so far) only be made to work
only when $X$ is isomorphic to a closed subspace of a space $L^q(\mu)$
with $q\in [2,\infty)$ (or a slight generalisation thereof, see \cite{NVW10b}).

In the following two subsections we compare (for stochastic equations) the theory presented
in this paper with the results in \cite{NVW11eq}.

\subsection{Part I}

Let us consider the example of Section \ref{sec:appl} (with $s=-1$) in more detail.
Initial values are taken in $L^q(\R^d)$ and the solutions are
in $ L^p(\O; H^{1,q}(\R^d; L^2(0,T)))$ for any $p\in (0,\infty)$.
Here, $q\in (1,\infty)$ may be chosen arbitrarily.
In contrast, the stochastic maximal $L^p$-regularity result of
\cite{NVW11eq} allows  initial values in $B^{\frac12-\frac1p}_{q,p}(\R^d)$ and then returns solutions
in $L^p(\O; L^p(0,T;H^{1,q}(\R^d)))$ for any $p\in (2,\infty)$. Here, we had to restrict to values
$q\in [2,\infty)$ ($p=2$ being allowed if $q=2$).

Thus we see that, in this example, the pathwise regularity in $L^p(0,T;H^{1,q}(\R^d))$ of \cite{NVW11eq} is replaced here
with pathwise regularity in $H^{1,q}(\R^d; L^2(0,T))$. The case $q\in (1,2)$ is not covered by
the results of \cite{NVW11eq}; here, for these exponents the underlying space has cotype $2$ and
therefore, for $q\in (1,2)$ we actually pick up pathwise regularity in $L^2(0,T;H^{1,q}(\R^d))$.
In the case $q=2$, both theories apply and prove pathwise regularity in $L^2(0,T;H^{1,2}(\R^d))$
for initial conditions in $L^2(\R^d)$.
 The results are summarized in the following table.

\medskip
\begin{center}
\begin{tabular}{|l|c|c|}
\hline
   & $\g$-theory	with $q\in (1,2]$	& $L^p$-theory with $p\in (2,\infty)$, $q\in [2,\infty)$
  \\ \hline
  initial value	& $u_0 \in L^q(\R^d)$ 		& $u_0 \in B^{\frac12-\frac1p}_{q,p}(\R^d)$
  \\ \hline
 pathwise regularity 	& $L^2(0,T;H^{1,q}(\R^d))$ & $L^p(0,T;H^{1,q}(\R^d))$
  \\ \hline
 trace regularity	& $C([0,T];L^q(\R^d))$	& $C([0,T];B^{\frac12-\frac1p}_{q,p}(\R^d))$
\\ \hline
\end{tabular}
\end{center}

\subsection{Part II}
In order two compare the theory here with the theory of \cite{NVW11eq} with the stochastic heat equation
on $\R^d$,
\[ \left\{
   \begin{array}{ll}
     du & = \Delta u +  B(u) \, d W, \\
     u(0,\cdot) & = u_0,
   \end{array}
 \right.
\]
where $\Delta$ is the Laplacian on $X=H^{s-1,q}(\R^d)$ with domain $X_1=H^{s+1,q}(\R^d)$.
We assume that $B:H^{s+1,q}(\R^d)\to H^{s,q}(\R^d)$ is given by
\[B(u)(x) = b(x) \cdot \nabla u\]
with $b\in C^\infty_b(\R^d)$. Finally $W:\R_+\times\O\to \R$ is a standard Brownian
motion and $u_0:\O\to S'(\R^d)$ is an $\F_0$-measurable initial value.

In order to make a good comparison with stochastic maximal $L^p$-regularity,
let us apply the results of \cite{NVW11eq} to the state space
$Y_0 = X_{\frac1p-\frac12}$, so that the trace space becomes $Y_{1-\frac1p,p}
= X_{\frac12,p}$.

As we have seen, for $u_0\in L^0(\O;X_{\frac12})$, the stochastic maximal $\g$-regularity result of Theorem \ref{thm:SE}
produces solutions $U$ which are pathwise in $\g^{\theta}(0,T;X_{1-\theta})$ for all
$\theta\in [0, \frac{1}{2})$. In particular, by Remark \ref{rem:improvedembedding}, pathwise one has $U\in H^{\theta+\frac1q - \frac12,q}(0,T;X_{1-\theta})$.
On the other hand, for $u_0\in Y_{1-\frac1p,p}$, the stochastic maximal $L^p$-regularity results of \cite{NVW11eq}
provide solutions $U$ which are pathwise in $H^{\theta',p}(0,T;Y_{1-\theta'})$ for all $\theta'\in [0, \frac{1}{2})$.
Choosing $\theta' =
\theta+\frac1q-\frac12$ leads to solutions $U$ pathwise
in $H^{\theta+\frac1q-\frac12,p}(0,T;X_{1+\frac1p-\frac1q - \theta}).$ Taking
$p=q$ (which is allowed if $q\in (2,\infty)$), this becomes
\[U\in H^{\theta+\frac1q-\frac12,q}(0,T;X_{1 - \theta}) \ \hbox{for} \  \theta\in [\tfrac12-\tfrac1q,\tfrac12).\]

A similar comparison can be made for the space regularity and the trace regularity in both cases.
The results are summarized in the following table.

\medskip
\begin{center}
\begin{tabular}{|l|c|c|}
\hline
   & $\g$-theory	with $q\in [2, \infty)$	& $L^p$-theory for $p\in (2,\infty)$, $q\in [2,\infty)$
  \\ \hline
  initial value	& $u_0 \in H^{s,q}(\R^d)$ 		& $u_0 \in B^{s}_{q,p}(\R^d)$
  \\ \hline
 pathwise regularity 	& $H^{\theta+\frac1q-\frac12,q}(0,T;H^{s+1-2\theta,q}(\R^d))$ & $H^{\theta',p}(0,T; H^{s+\frac2p -2\theta',q}(\R^d))$
  \\ \hline
 trace regularity	& $C([0,T];H^{s,q}(\R^d)))$      & $C([0,T];B^{s}_{q,p}(\R^d))$
\\ \hline
\end{tabular}
\end{center}
Here $\theta,\theta'\in [0,\tfrac12)$.

\medskip

The main smoothness exponents are comparable. However, there is a trade-off:
\begin{itemize}
\item The space regularity holds with smoothness exponent $s+1-2\theta$ for stochastic maximal $\g$-regularity versus $s+\frac2p -2\theta'$ for stochastic maximal $L^p$-regularity.
\item The time regularity holds with smoothness exponent $\theta+\frac1q-\frac12$ for stochastic maximal $\g$-regularity versus $\theta'$ for stochastic maximal $L^p$-regularity.
\end{itemize}
Summarizing,
one might say that for $L^q(\R^d)$-spaces with $q\in [2, \infty)$, stochastic
maximal $\g$-regularity gives more space-regularity and less time regularity
than stochastic maximal $L^p$-reg\-ularity.

\def\polhk#1{\setbox0=\hbox{#1}{\ooalign{\hidewidth
  \lower1.5ex\hbox{`}\hidewidth\crcr\unhbox0}}} \def\cprime{$'$}

\end{document}